\documentclass[12pt]{article}
\nonstopmode
\RequirePackage[colorlinks,citecolor=blue,urlcolor=blue,linkcolor=blue]{hyperref}
\usepackage{graphicx,xspace,colortbl}
\usepackage{amsmath,amsthm,amsfonts}
\usepackage{color}
\usepackage{fancybox}
\usepackage{epsfig}
\usepackage{subfig}
\usepackage{pdfsync}
    \def\qed{\hfill$\sqcap\kern-8.0pt\hbox{$\sqcup$}$\\}
    \def\beq{\begin{eqnarray}}
    \def\eeq{\end{eqnarray}}
    \def\beqq{\begin{eqnarray*}}
    \def\eeqq{\end{eqnarray*}}

    \def\re{\textnormal {Re}}
    \def\im{\textnormal {Im}}
    \def\p{{\mathbb P}}
    \def\e{{\mathbb E}}
    \def\r{{\mathbb R}}
    \def\c{{\mathbb C}}
    \def\pp{{\textnormal p}}	
    \def\d{{\textnormal d}}
    \def\i{{\textnormal i}}
    
    \def\plus{{\scriptscriptstyle +}}
    \def\minus{{\scriptscriptstyle -}}
    
    \def\phiqp{\phi_q^{\plus}}
    \def\phiqm{\phi_q^{\minus}}
    
    \def\ind{{\mathbb I}}

    \def\ll{{\mathcal L}} 
    \def\hh{{\mathcal H}}
    \def\pp{{\mathcal P}} 	
    \def\gg{{\mathcal G}} 
    \def\vv{{\textnormal v}} 
    \def\ee{{\textnormal e}} 
    \def\coeffa{{\textnormal a}} 
    \def\coeffb{{\textnormal b}} 
     
    \def\aa{{\mathbf A}} 
     
    \def\bb{{\mathbf B}} 
    \def\dd{{\mathbf D}} 
    \def\cc{{\mathbf C}} 
    \def\ii{{\mathbf I}} 
    \def\xsup{{\overline X}} 
    \def\xinf{{\underline X}} 
\newtheorem{theorem}{Theorem}
\newtheorem{lemma}{Lemma}
\newtheorem{proposition}{Proposition}
\newtheorem{corollary}{Corollary}
\theoremstyle{definition}
\newtheorem{definition}{Definition}

\newtheorem{remark}{Remark}

\title{
\textbf{Meromorphic L\'evy processes and their fluctuation identities}
}
\author{
\textbf{
A. Kuznetsov
\footnote{Department of Mathematics and Statistics, 
York University, 
4700 Keele Street, 
Toronto, Ontario, 
M3J 1P3, Canada. Email: kuznetsov@mathstat.yorku.ca}
}
,\, 
\textbf{
A. E. Kyprianou
\footnote{Department of Mathematical Sciences, University of Bath, Claverton Down, Bath, BA2 7AY, U.K.. Email: a.kyprianou@bath.ac.uk}
}, 
\, 
\textbf{J. C. Pardo
\footnote{Centro de Investigaci\'on en Matem\'aticas A.C. Calle Jalisco s/n. 36240 Guanajuato, M\'exico. Email: jcpardo@cimat.mx}
} 
}
\date{\footnotesize This version: \today}

\begin{document}
\maketitle
\begin{abstract}
The last couple of years has seen a remarkable number of new, explicit examples of the Wiener-Hopf factorization for L\'evy processes where previously there had been very few.  We mention in particular the many cases of spectrally negative L\'evy processes in \cite{HuKy, KyRi}, hyper-exponential  and generalized hyper-exponential L\'evy processes \cite{JP}, Lamperti-stable processes in \cite{CaCha, Caballero2008,  Chaumont2009, Patie2009}, Hypergeometric processes in \cite{KyPaRi, KyPavS, Caballeroetal2009}, $\beta$-processes in \cite{Kuz-beta} and $\theta$-processes in \cite{Kuz-theta}.

In this paper we introduce a new family of L\'evy processes, which we call Meromorphic L\'evy processes, or just $M$-processes for short, which overlaps with many of the aforementioned classes. A key feature of the $M$-class is the identification of their Wiener-Hopf factors as rational functions of infinite degree written in terms of poles and roots of the Laplace exponent, all of which are real numbers. The specific structure of the $M$-class Wiener-Hopf factorization enables us to explicitly handle a comprehensive suite of fluctuation identities that concern first passage problems for finite and infinite intervals for both the process itself as well as the resulting process when it is reflected in its infimum. Such identities are of fundamental interest given their repeated occurrence in various fields of applied probability such as mathematical finance, insurance risk theory and queuing theory.

\bigskip

\noindent {\it Keywords:} L\'evy processes, Wiener-Hopf factorization, exit problems, fluctuation theory.

\medskip

\noindent{\it AMS 2000 subject classifications: 60G51, 60G50.}
\end{abstract}


\section{Introduction}
The theory of L\'evy processes forms the cornerstone of an enormous volume of mathematical literature which supports a wide variety of applied and theoretical stochastic models. 
One of the most obvious and fundamental problems that can be stated for a L\'evy process, particularly in relation to its role as a modelling tool, is the distributional characterization of the time at which a L\'evy process first exists either an infinite or finite interval together with its overshoot beyond the boundary of the interval.
 As a family of stochastic processes, L\'evy processes are now well understood  and the exit problem has seen many different approaches dating back to the 1960s. See for example \cite{Avrametal, Bertoin1997, Borovkov, Doney, Rogozin, Getoor, Chaumont2009, KyPaRi} to name but a few.
 
 Despite the maturity of this field of study it is surprising to note that, until very recently, there were less than a handful of examples for which explicit analytical detail concerning the first exit problem could be explored. Given the closeness in mathematical proximity of the first exit problem to the characterization of the Wiener-Hopf factorization, one might argue that the lack of concrete examples of the former was a consequence of the same being true for the latter.  The landscape for both the Wiener-Hopf factorization problem and the first exit problem has changed quite rapidly in the last couple of years however with the discovery of a number of new, mathematically tractable families of L\'evy processes.  We mention in particular the many cases of spectrally negative L\'evy processes in \cite{HuKy, KyRi}, hyper-exponential  and generalized hyper-exponential L\'evy processes \cite{JP}, Lamperti-stable processes in \cite{CaCha, Caballero2008, Chaumont2009}, Hypergeometric processes in \cite{KyPaRi, KyPavS}, $\beta$-processes in \cite{Kuz-beta} and $\theta$-processes in \cite{Kuz-theta}.

In this paper we introduce a new family of L\'evy processes, which we call Meromorphic L\'evy processes, or just $M$-class for short, that overlaps with many of the aforementioned classes. Our definition of the $M$-class of processes will allow us to drive features of their Wiener-Hopf factors  through to many of the fluctuation identities which are of pertinence for a wide variety of applications.
In the theory of actuarial mathematics, the problem of  first exit from a half-line is of fundamental interest with regard to the classical ruin problem and is typically studied within the context of an expected discounted penalty function. The latter is also known as the Gerber-Shiu function following the first article \cite{MR1604928} of a long series of papers found within the actuarial literature. In the setting of financial mathematics, the first exit of a L\'evy process, as well as a L\'evy process reflected in its infimum, from an interval is of interest in the pricing of barrier options and American-type options,  \cite{Alili2005}, as well as certain credit risk models, \cite{hilberink, surya}. In queueing theory  exit problems for L\'evy processes play a central role in understanding  the  trajectory of the workload during busy periods as well as in relation to buffers,  \cite{mandjes, konstetal}. Many optimal stopping strategies also turn out to boil down to first  passage problems for both L\'evy processes and L\'evy processes reflected in their infimum; classic examples of which include McKean's optimal stopping problem, \cite{McKean}, as well as the Shepp-Shiryaev optimal stopping problem, \cite{Avrametal, BaKy} .

 It is not our purpose however to dwell on these applications. As alluded to above, the main objective will be expose a comprehensive suite of fluctuation identities in explicit form for the $M$-class of L\'evy processes. We thus conclude the introduction with an overview of the paper. 
 
 \bigskip
 
In the next section we give a formal definition of meromorphic L\'evy processes, dwelling in particular on their relationship with discrete completely monotone functions. In Section \ref{examples} we consider several classes of existing families of L\'evy processes that have appeared in recent literature. Next, in Section \ref{one-sided} we establish explicit identities for the exponentially discounted first passage problem. In particular we deal with (what is known in the actuarial literature as) the Gerber-Shiu measure describing the discounted joint triple law of the overshoot, undershoot and undershoot of the maximum at first passage over a level as well as the marginal thereof which specifies the law of the discounted overshoot.  It is important to note that the discounting factor which appears in all of the identities means that one is never more than a Fourier transform away from the naturally associated space-time identity in which the additional law of the time to first passage is specified. 
This last Fourier inversion appears to be virtually impossible to produce analytically within the current context, but the expressions we offer are not difficult to work with in conjunction with straightforward Fourier inversion algorithms. 

In Section \ref{sec_dble_exit} we look at the more complicated two-sided exit problem. Inspired by a technique of Rogozin \cite{Rogozin} we  solve a system of equations which characterize the discounted overshoot distribution on either side of the interval in question. The same technique also delivers explicit expressions for the discounted entrance law into an interval. In Section \ref{ladder} we look in analytical detail at what can be said of the ascending and descending ladder variables. In particular we offer expressions for their joint Laplace exponent and jump measure. Section \ref{more} mentions some additional examples of fluctuation identities which enjoy explicit detail and finally in Section \ref{sec_numerics} we describe some numerical experiments in order to exhibit the genuine practical applicability of our method.

\section{Meromorphic L\'evy processes}

Recall that a general one-dimensional L\'evy process is a stochastic process issued from the origin with stationary and independent increments and almost sure right continuous paths. We write $X=\{X_t : t\geq 0\}$ for its trajectory and $\mathbb{P}$ for its law. 
As $X$ is necessarily a strong Markov process, for each $x\in\mathbb{R}$, we will need the probability 
$\mathbb{P}_x$  to denote the law of $X$ when issued from $x$ with the understanding that $\mathbb{P}_0 = \mathbb{P}$. The law $\mathbb{P}$ of L\'evy processes is characterized by its one-time transition probabilities. In particular there always exist a  triple $(a, \sigma, \Pi)$ where $a\in\mathbb{R}$, $\sigma\in \mathbb{R}$ and $\Pi$ is a measure on $\mathbb{R}\backslash\{0\}$ satisfying the integrability condition $\int_{\mathbb{R}}(1\wedge x^2)\Pi({\rm d}x)<\infty$, such that, for all $z \in\mathbb{R}$ 
\beq\label{eq1}
\mathbb{E}[e^{\i z X_t}] = e^{ t \psi(\i z)},
\eeq
where the Laplace exponent $\psi(z)$ is given by the L\'evy-Khintchine formula
\beq\label{def_psi}
\psi(z) = \frac{1}{2}\sigma^2 z^2 +  \mu z + \int_{\mathbb{R}}\left( e^{ z x} -1- z x h(x)\right)\Pi({\rm d}x).
\eeq
Here $h(x)$ is the cutoff function which is usually taken to be $h(x)\equiv  \ind_{\{|x|<1\}}$.  However, everywhere in the present paper we will work
with the L\'evy measures $\Pi(\d x)$ which have exponentially decaying tails, thus we will take $h(x)\equiv 1$.  Note, that the exponential decay of the tails of the L\'evy measure also implies that the Laplace exponent can be analytically continued into some vertical strip 
$a<\re(z)<b$ for $a<0<b$. 

Everywhere in this paper we will denote the tails of the L\'evy measure as $\bar \Pi^+(x)=\Pi((x,\infty))$ and $\bar \Pi^-(x)=\Pi((-\infty,-x))$ for $x>0$.  Let us define the supremum/infimum processes $\overline{X}_t = \sup_{s\leq t}X_s$, $\underline{X}_t= \inf_{s\leq t}X_s$ and denote by $\ee(q)$ an independent, exponentially distributed random variable with rate $q>0$. Finally, the first passage time above $x$ is defined as
$\tau_x^+=\inf\{t>0: X_t>x\}$, and similarly $\tau_y^-=\inf\{t>0: X_t<y\}$. 

Let us recall some basic facts about completely monotone functions. A function $f: (0,\infty) \mapsto \r$ is called {\it completely monotone} if $f \in {\mathcal C}^{\infty}$  and 
$(-1)^n f^{(n)}(x)\ge 0$ for all $n\in {\mathbb N}\cup\{0\}$ and $x>0$. According to 
Bernstein theorem (see Theorem 1.4 in \cite{Schilling}), the function $f$ is completely monotone if and only if it can be represented as
the Laplace transform of a positive measure on $[0,\infty)$:
\beq\label{Bernstein_thm}
f(x)=\int\limits_{[0,\infty)} e^{-zx} \mu (\d z), \;\;\; x>0.
\eeq 
Note that if $f(0^+)=1$, then $\mu(\d x)$ is a probability measure and $1-f(x)$ is the cumulative distribution function of a positive infinitely divisible 
random variable $Z$, whose distribution is a mixture of exponential distributions. We will denote the class of completely monotone functions as ${\mathcal {CM}}$. 

 Next, let us introduce a subclass of completely monotone functions, which 
will be important for us later. We will call  $f: (0,\infty)\mapsto \r$ a {\it discrete completely monotone} function if the measure $\mu(\d z)$ in the representation
(\ref{Bernstein_thm}) is discrete, and the support of the measure $\mu(\d z)$ is infinite and does not have finite accumulation points. This implies that the measure $\mu(\d z)$ is an infinite mixture of atoms of size $a_n$ at the points $b_n$
\beqq
\mu(\d z)=\sum\limits_{n\ge 1} a_n \delta_{b_n}(\d z),
\eeqq
where $\delta_{b}(\d x)$ denotes the Dirac measure at $x=b$, for all $n\ge 1$ we have $a_n>0$, $b_n\ge 0$ and $b_n \to \infty$ as $n\to \infty$ 
(without loss of generality we can also assume that the sequence $\{b_n\}_{n\ge 1}$ is strictly increasing). 
From (\ref{Bernstein_thm}) it follows then that any discrete completely monotone function 
can be represented as an infinite series of exponential functions
\beq
f(x)=\sum\limits_{n\ge 1} a_n e^{-b_n x}, \;\;\; x>0. 
\eeq
We will denote the class of discrete completely monotone functions as ${\mathcal {DCM}}$.

\begin{definition}[$M$-class]\label{maindef} A L\'evy process $X$ is said to belong to the Meromorophic class ($M$-class) if 
$\bar \Pi^+(x), \bar \Pi^-(x) \in {\mathcal {DCM}}$.
\end{definition}

We see that according to our definition of the discrete completely monotone functions, the process is 
Meromorphic if and only if the L\'evy measure $\Pi(\d x)$ has
a density with respect to the Lebesgue measure, given by
\beq\label{def_pi}
\pi(x)=\ind_{\{x>0\}} \sum\limits_{n\ge 1} a_n \rho_n e^{-\rho_n x}+ \ind_{\{x<0\}} \sum\limits_{n\ge 1} \hat a_n \hat \rho_n e^{\hat \rho_n x},
\eeq 
 where all the coefficients  $a_n$, $\hat a_n$, $\rho_n$, $\hat \rho_n$ are positive, 
the sequences $\{\rho_n\}_{n \ge 1}$ and $\{\hat \rho_n\}_{n\ge 1}$ are stricly increasing, and $\rho_n \to +\infty$ and $\hat \rho_n \to +\infty$ as $n\to +\infty$.

\begin{proposition}\label{integrability}
Assume that $\pi(x)$ is given by (\ref{def_pi}).
 The integral $\int_{\r} x^2 \pi(x) \d x$ converges if and only if both series $\sum_{n\ge 1} a_n \rho_n^{-2}$ and 
$\sum_{n\ge 1} \hat a_n \hat \rho_n^{-2}$ converge.
\end{proposition}
\begin{proof}
The "if" part  was established in Proposition 1 in \cite{Kuz-theta}. The "only if" part follows from (\ref{def_pi}) and the Monotone Convergence Theorem. 
\end{proof}

Recall, that a function $g: \c \mapsto \c \cup \{\infty\}$ is called {\it meromorphic} if it does not have any other singularities in the open complex plane except for poles. A function $g(z)$ is  called {\it real meromorphic} function if it is meromorphic and $g(z) \in \r \cup \{\infty\}$ for all $z\in \r$, 
or equivalently, if ${\overline{g(z)}}=g(\bar z)$.  

\begin{theorem}\label{thm_equivalent}
 The following conditions are equivalent:
\begin{itemize}
 \item[(i)]   $X$ is Meromorphic.
 \item[(ii)]  $\bar \Pi^+(x), \bar \Pi^-(x) \in {\mathcal {CM}}$ and the Laplace exponent $\psi(z)$ is meromorphic.
 \item[(iii)] For some $q>0$ (and then, for every $q>0$) the functions $\p(\xsup_{\ee(q)}>x)$ and $\p(-\xinf_{\ee(q)}>x)$ restricted to $x>0$ belong to the class ${\mathcal {DCM}}$.
 \item[(iv)]  For some $q>0$ (and then, for every $q>0$) the functions $\p(X_{\ee(q)}>x)$ and $\p(-X_{\ee(q)}>x)$ restricted to $x>0$ belong to the class ${\mathcal {DCM}}$.
 \item[(v)]   For some $q>0$ (and then, for every $q>0$) we have the factorization
 \beq\label{WH_fact_1}
 q-\psi(z)=q\prod\limits_{n\ge 1} \frac{1-\frac{z}{\zeta_n}}{1-\frac{z}{\rho_n}}
\prod\limits_{n\ge 1} \frac{1+\frac{z}{\hat \zeta_n}}{1+\frac{z}{\hat \rho_n}}, \;\;\; z\in \c,
 \eeq
where all roots $\{\zeta_n, - \hat \zeta_n\}$ of $\psi(z)-q$ are real and interlace with the poles $\{\rho_n, - \hat \rho_n\}$
 \beq\label{interlacing_prop}
 ... -\hat\rho_2 <-\hat\zeta_2<-\hat\rho_1<- \hat\zeta_1 <0 < \zeta_1 <  \rho_1 < \zeta_2 < \rho_2 < ...
 \eeq
 \item[(vi)] The Laplace exponent $\psi(z)$ is a real meromorphic function, which satisfies $\im(\psi(z)/z)>0$ for all $z$ in the half-plane $\im(z)>0$.
\end{itemize}
\end{theorem}
\begin{proof}
The main ideas and tools needed for the proof of this theorem come from the proofs of Theorem 2 in \cite{Rogers01071983} and Theorem 1 in \cite{Kuz-theta}. 

\noindent
{\it (i)$\Rightarrow$(ii)}  Using (\ref{def_pi}) and (\ref{def_psi}) we find that 
 \beq\label{eq_psi}
 \psi(z)=\frac 12 \sigma^2 z^2 + \mu z
 +z^2 \sum\limits_{n\ge 1}\frac{a_n }{\rho_n(\rho_n- z)}+z^2 \sum\limits_{n\ge 1} \frac{\hat a_n}{\hat \rho_n (\hat \rho_n+z)}
, \;\;\; z\in \c,
 \eeq
which shows that $\psi(z)$ is a meromorphic function. Since ${\mathcal {DCM}}\subset {\mathcal {CM}}$ this proves (ii).

\noindent
{\it (ii)$\Rightarrow$(vi)} We know that $\psi(z)$ is a meromorphic function, and that $\psi(0)=0$, therefore $\psi(z)$ is analytic in some neighbourhood
of zero. From the proof of Theorem 2 in \cite{Rogers01071983} we find that (a) $\psi(z)$ can be analytically continued in the half-planes $\im(z)>0$ and $\im(z)<0$, 
(b) $\im(\psi(z)/z)>0$ for all $z$ in the half-plane $\im(z)>0$. 
Since $\overline{ \psi(z)}=\psi(\bar z)$ we conclude that  $\psi(z)$ is a real meromorphic function. Using this fact and the above statement (b) we obtain (vi).

\noindent
 {\it (vi)$\Rightarrow$(v),  $\forall q>0$.} For $q>0$ and $\im(z)>0$ it is true that $\im(-q/z)>0$, thus (vi) implies that $\im((\psi(z)-q)/z)>0$ for all $z$ in the half-plane $\im(z)>0$. Using Theorem 1 on page 220 in  \cite{Levin1996} 
(the statement of this result can also be found in the proof of Theorem 1 in \cite{Kuz-theta}) we find that (v) is valid for all $q>0$.

\noindent
{\it (v)$\Rightarrow$(iii)} This follows from Theorem 1 in \cite{Kuz-theta}.

\noindent
{\it (iii)$\Rightarrow$(iv)} One can also check that 
if two functions $f(x)$ and $g(-x)$ belong to the class ${\mathcal {DCM}}$, then the same is true for $\ind_{\{x>0\}}(f*g)(x)$ and  $\ind_{\{x<0\}}(f*g)(x)$, where $f*g$ is the convolution. The result (iv) then follows easily from this fact and the Wiener-Hopf decomposition which says that $X_{\ee(q)}\stackrel{d}{=} Y_1+ Y_2$,
where the random variables $Y_1$ and $Y_2$ are independent, $Y_1\stackrel{d}{=} \xsup_{\ee(q)}$ and $Y_2\stackrel{d}{=} \xinf_{\ee(q)}$. 

\noindent
{\it (iv)$\Rightarrow$(i)} From (iv) and the definition of the ${\mathcal {DCM}}$ class 
we know that for some $q>0$ there exist $\alpha_0\ge 0$ and positive constants $\alpha_n$, $\beta_n$, $\hat \alpha_n$, $\hat \beta_n$ such that
\beq\label{u^q}
\p(X_{\ee(q)}\in \d x)=\alpha_0 \delta_0(\d x)+ \left[\ind_{\{x>0\}} \sum\limits_{n\ge 1} \alpha_n \beta_n^{-1} e^{-\beta_n x}+ \ind_{\{x<0\}} \sum\limits_{n\ge 1} \hat \alpha_n 
\hat \beta_n^{-1}  e^{\hat \beta_n x} \right] \d x.
\eeq  
Note that the condition $\p(X_{\ee(q)}\in \r)=1$ implies that both series $\sum_{n\ge 1} \alpha_n \beta_n^{-2}$ and $\sum_{n\ge 1} \hat \alpha_n \hat \beta_n^{-2}$ converge.
Let us define $\tilde \psi(z)=z^2 \e[\exp(z X_{\ee(q)})]$. Using (\ref{u^q}) we obtain
\beqq
\tilde \psi(z)=\alpha_0 z^2 + z^2 \sum\limits_{n\ge 1}\frac{\alpha_n }{\beta_n(\beta_n- z)}+z^2 \sum\limits_{n\ge 1} \frac{\hat \alpha_n}{\hat \beta_n(\hat \beta_n+z)}
, \;\;\; z\in \c.
\eeqq
Comparing the above formula with (\ref{eq_psi}) we conclude that $\tilde \psi(z)$ is a Laplace exponent of a 
Meromorphic L\'evy process. We have already proved that (i) implies (vi), therefore $\im(\tilde \psi(z)/z )>0$ for all $z$ in the half-plane $\im(z)>0$. 

Next, using the definition of the Laplace exponent (\ref{eq1}) it is easy to verify that $\e[\exp(z X_{\ee(q)})]=q/(q-\psi(z))$. Therefore 
\beqq
\frac{\psi(z)-q}{z}=-\frac{qz}{\tilde \psi(z)}.
\eeqq 
As we have already established,  $\im(\tilde \psi(z)/z )>0$ for all $z$ in the half-plane  $\im(z)>0$. Using this fact and the above identity we conclude that
 $\im((\psi(z)-q)/z )>0$. Applying Theorem 1 on page 197 in \cite{Chebotarev1949} 
(the statement of this result can also be found in the proof of Theorem 1 in \cite{Kuz-theta}) we find that $\psi(z)$ admits a representation of the form (\ref{eq_psi}), which in turn implies that the process $X$ is meromorphic. 

Finally, note that if (v) is true for some $q>0$, then (iii) and (iv) are valid for the same value of $q>0$. But as we have already demonstrated, any of the conditions
(iii), (iv), (v) implies (i) which is equivalent to (v) being valid for all $q>0$. Thus, if one of the conditions (iii), (iv), (v)  is valid for some $q>0$, 
then it must be valid for all $q>0$.  
\end{proof}

Statement (ii) in Theorem \ref{thm_equivalent} shows that the $M$-class of L\'evy processes might also be called 
"processes with completely-monotone L\'evy measure and meromorphic Laplace exponent", which explains the origin of the name "Meromorphic L\'evy processes". 
Note, however, that there exist L\'evy processes with meromorphic Laplace exponent but not completely monotone L\'evy measure.

The following technical result on partial fraction decomposition of infinite products will be very important for us later. 

\begin{lemma}\label{analytical_lemma}
 Assume that we have two increasing sequences $\rho=\{\rho_n\}_{n\ge 1}$ and $\zeta=\{\zeta_n\}_{n \ge 1}$ of positive numbers, such that
 $\rho_n\to +\infty$ as $n\to +\infty$ and the following interlacing condition is satisfied:
 \beq\label{interlacing}
 \zeta_1< \rho_1 < \zeta_2 < \rho_2 < ...
 \eeq
 Define
 \beq\label{def_phi}
\phi(z)= \prod\limits_{n\ge 1}  \frac{1+\frac{z}{\rho_n}}{1+\frac{z}{\zeta_n}}, \;\;\; z>0.
\eeq
Then for all $z>0$
 \beq\label{part_fraction1}
  \phi(z) &=&\coeffa_0(\rho,\zeta)+
\sum\limits_{n\ge 1} \coeffa_n(\rho,\zeta) \frac{\zeta_n}{\zeta_n+z}, \\
 \label{part_fraction2}
 \frac{1}{\phi(z)} &=&1+z\coeffb_0(\zeta,\rho)+
\sum\limits_{n\ge 1} \coeffb_n(\zeta,\rho) \left[1-\frac{\rho_n}{\rho_n+z}\right],
 \eeq 
where 
 \beq\label{formula_cn_a_b}
 \coeffa_0(\rho,\zeta)&=&\lim_{n\to +\infty}\prod\limits_{k=1}^n \frac{\zeta_k}{\rho_k}, \;\;\;
\coeffa_n(\rho,\zeta)=\left(1-\frac{\zeta_n}{\rho_n}\right) \prod\limits_{\substack{k\ge 1 \\ k\ne n}}  \frac{1-\frac{\zeta_n}{\rho_k}}{1-\frac{\zeta_n}{\zeta_k}}, \\
\label{formula_cn_b_a}
 \coeffb_0(\zeta,\rho)&=&\frac{1}{\zeta_1}\lim_{n\to +\infty}\prod\limits_{k=1}^n \frac{\rho_k}{\zeta_{k+1}}, \;\;\;
\coeffb_n(\zeta,\rho)=-\left(1-\frac{\rho_n}{\zeta_n}\right) \prod\limits_{\substack{k\ge 1 \\ k\ne n}} 
 \frac{1-\frac{\rho_n}{\zeta_k}}{1-\frac{\rho_n}{\rho_k}}.
 \eeq
 Moreover, $\coeffa_0(\rho,\zeta)\ge 0$, $\coeffb_0(\zeta,\rho)\ge 0$ and for all $n\ge 1$ we have $\coeffa_n(\rho,\zeta) > 0$, $\coeffb_n(\zeta,\rho) > 0$.
\end{lemma}
\begin{proof}
The convergence of the infinite product in (\ref{def_phi}) and the partial fraction decomposition (\ref{part_fraction1}) follow from 
Theorem 1 in \cite{Kuz-theta}. To prove the second partial fraction decomposition, rewrite the infinite product as
 \beqq
 \left(1+\frac{z}{\zeta_1} \right) \prod\limits_{n\ge 1}  \frac{1+\frac{z}{\zeta_{n+1}}}{1+\frac{z}{\rho_n}},
 \eeqq
and note that sequences $\{\zeta_{n+1}\}_{n\ge 1}$ and $\{\rho_n\}_{n\ge 1}$ satisfy interlacing condition, thus we can apply the same method. The details are left to the reader.
\end{proof}

Using Monotone Convergence Theorem one can show that formulas (\ref{part_fraction1}) and (\ref{part_fraction2}) are equivalent to 
\beq\label{see-the-atom}
 \phi(z)&=& \coeffa_0(\rho,\zeta) + \int\limits_{\r^+} \left[ \sum\limits_{n\ge 1} \coeffa_n(\rho,\zeta) \zeta_n e^{-\zeta_n x} \right] e^{-zx} \d x,\\
\frac{1}{ \phi(z)}&=& 1+z \coeffb_0(\zeta,\rho) +\int\limits_{\r^+} \left[\sum\limits_{n\ge 1} \coeffb_n(\zeta,\rho) 
 \rho_n e^{-\rho_n x} \right] \left(1-e^{-zx}  \right) \d x.
\label{see-the-drift}
\eeq

Before stating our next result we recall that a L\'evy process creeps upwards if for some (and then all) $x\geq 0$, $\mathbb{P}(X_{\tau^+_x} = x|\tau^+_x <\infty) >0$. Moreover, we say that $0$ is irregular for $(0,\infty)$ if and only if $\mathbb{P}(\tau^+_0>0)=1$. (Note that this probability can only be 0 or 1 thanks to the Blumenthal zero-one law). We refer to Bertoin \cite{Bertoin}, Doney \cite{MR2320889} or Kyprianou \cite{Kyprianou} for more extensive discussion of these subtle path properties.

We also need to introduce some more notation. In the forthcoming text everything will depend on the coefficients   $\{\coeffa_n(\rho,\zeta), \coeffa_n(\hat \rho,\hat \zeta)\}_{n\ge 0}$ 
defined using (\ref{formula_cn_a_b})
and $\{\coeffb_n(\zeta,\rho), \coeffb_n(\hat \zeta,\hat \rho)\}_{n\ge 0}$ 
defined using (\ref{formula_cn_b_a}).  We define for convenience a column vector
 \beqq
 \bar \coeffa(\rho,\zeta)=\left[  \coeffa_0(\rho,\zeta), \coeffa_1(\rho,\zeta),\coeffa_2(\rho,\zeta), ... \right]^T,
 \eeqq
and similarly for $\bar \coeffa(\hat \rho,\hat \zeta)$,  $\bar \coeffb(\zeta,\rho)$ and $\bar \coeffb(\hat \zeta,\hat \rho)$.
 Next, given a sequence of positive numbers $\zeta=\{\zeta_n\}_{n\ge 1}$,  we define a column vector $\bar \vv(\zeta,x)$ as a vector of distributions
\beqq
\bar \vv(\zeta,x)=\left[ \delta_0(x) , \zeta_1 e^{-\zeta_1 x}, \zeta_2 e^{-\zeta_2 x},\dots \right] ^T ,
\eeqq
where $\delta_0(x)$ is the Dirac delta function at $x=0$.

\bigskip

{\bf From here on, unless otherwise stated, we shall always assume that $X$ is a L\'evy process belonging to the $M$-class  and that $X$ is not a compound Poisson
process.}

\noindent
\begin{theorem}[Properties of Meromorphic processes]\label{M_class_properties} 
${}$
\\
\begin{itemize}
 \item[(i)] The Wiener-Hopf factors are given by
\beq\label{eqn_WH_factor}
 \phiqp(\i z)=\e \left[ e^{-z \xsup_{\ee(q)}} \right]=\prod\limits_{n\ge 1} \frac{1+\frac{z}{\rho_n}}{1+\frac{z}{\zeta_n}}, \;\;\; \qquad
 \phiqm(-\i z)=\e \left[ e^{z \xinf_{\ee(q)}} \right]=\prod\limits_{n\ge 1} \frac{1+\frac{z}{\hat \rho_n}}{1+\frac{z}{\hat \zeta_n}}.
\eeq
 \item[(ii)] For $x \ge 0$
\beq\label{dist_sup_X}
 \p(\xsup_{\ee(q)} \in \d x) =   \bar\coeffa(\rho,\zeta)^T \times \bar\vv(\zeta,x) \d x, \;\;\;
\p(-\xinf_{\ee(q)} \in \d x) =   \bar\coeffa(\hat \rho,\hat \zeta)^T \times \bar\vv(\hat \zeta,x) \d x.
\eeq
\item[(iii)] $\coeffa_0(\rho,\zeta)$ {\em ($\coeffa_0(\hat \rho,\hat \zeta)$)} is nonzero if and only if $0$ is irregular for $(0,\infty)$ 
  {\em (correspondingly, $(-\infty,0)$)}.

 \item[(iv)] $\coeffb_0(\zeta,\rho)$ {\em ($\coeffb_0(\hat \zeta,\hat \rho)$)} is nonzero if and only if the process $X$ creeps upwards 
 {\em (correspondingly, downwards)}.

 \item[(v)] For every $q>0$
 \beq\label{u^q_formula}
 \p(X_{\ee(q)}\in \d x)=q\left[\ind_{\{x>0\}} \sum\limits_{n\ge 1} \frac{e^{-\zeta_n x}}{\psi'(\zeta_n)} - \ind_{\{x<0\}} \sum\limits_{n\ge 1} 
 \frac{e^{\hat \zeta_n x}}{\psi'(-\hat \zeta_n)} \right] \d x.
 \eeq
\end{itemize}
\end{theorem}
\begin{proof}
Parts (i) and (ii) were established in Theorem 1 in \cite{Kuz-theta}, these results also follow easily from formulas (\ref{WH_fact_1}), (\ref{see-the-atom}) and the structure of the Wiener-Hopf factorization. Let us prove (iii).  
We note that $0$ is irregular for $(0,\infty)$ if and only if, for any $q>0$, $\overline{X}_{\ee(q)}$ has an atom in its distribution at $0$. From part (ii) this is clearly the case if and only if $\coeffa_0(\rho,\zeta)$ is non-zero.
For part (iv) of the proof, we note that we may necessarily write 
\[
 \phi^+_q(\i z) = \frac{\kappa(q,0)}{\kappa(q,z)} ,
\]
where $\kappa(q,z)$ is the Laplace exponent of the bivariate subordinator which describes the ascending ladder process of $X$. It is also known that a L\'evy process creeps upwards if and only if the subordinator describing its ladder height process has a linear drift component. The drift coefficient is then described by $\lim_{z\uparrow\infty} \kappa(q,z)/z\in[0,\infty)$ where the limit is independent of the value of  $q\geq 0$. Inspecting (\ref{see-the-drift}) we see that the required drift coefficient is non-zero if and only  if $\coeffb_0(\zeta,\rho)$ is non-zero. The conclusion thus follows.

Finally, let us prove (v). From part (iv) of Theorem \ref{thm_equivalent} we know that $\p(X_{\ee(q)}\in \d x)$ can be written in the form 
(\ref{u^q}), where $\alpha_0=0$ due to our assumption that $X$ is not a compound Poisson process. Combining this fact and the identity $\e[\exp(z X_{\ee(q)})]=q/(q-\psi(z))$ we conclude that
\beq\label{u^q_find_residues}
 \frac{q}{q-\psi(z)}=\sum\limits_{n\ge 1}\frac{\alpha_n }{\beta_n(\beta_n- z)}+\sum\limits_{n\ge 1} \frac{\hat \alpha_n}{\hat \beta_n(\hat \beta_n+z)}, \;\;\; z\in \c. 
\eeq
We see that the function on the right-hand side of the above equation has poles at the points $\beta_n$ and $-\hat \beta_n$, while the left hand 
side has poles at $\zeta_n$ and $-\hat \zeta_n$ (this follows from (\ref{WH_fact_1})). Therefore we conclude that $\beta_n=\zeta_n$ and $\hat \beta_n=\hat \zeta_n$. Comparing the residues of both sides of (\ref{u^q_find_residues}) at the pole $z=\zeta_n$ we see that
\beqq
\frac{\alpha_n}{\beta_n}=-{\textrm{Res}}\left( \frac{q}{q-\psi(z)} : \; z=\zeta_n \right)=\frac{q}{\psi'(\zeta_n)}.
\eeqq
Similarly we find 
\beqq
\frac{\hat\alpha_n}{\hat\beta_n}={\textrm{Res}}\left( \frac{q}{q-\psi(z)} : \; z=-\hat \zeta_n \right)=-\frac{q}{\psi'(-\hat\zeta_n)}.
\eeqq
Combining the above two identities and (\ref{u^q}) we obtain (\ref{u^q_formula}).
\end{proof}

We would like to emphasize the importance of statement (v) of Theorem \ref{M_class_properties}. It is well known that there exist only a few specific 
examples of L\'evy processes (such as Variance Gamma or Normal Inverse Gaussian processes) for which the law of $X_t$ is known explicitly (for every $t>0$). 
 While we do not know the law of $X_t$ for Meromorphic processes, the result in Theorem \ref{M_class_properties} (v) shows that Meromorphic processes 
have an advantage that at least the distribution of $X_{\ee(q)}$ can be easily computed (for every $q>0$). 
The formula is given in terms of the roots $\zeta_n$ and $\hat \zeta_n$, but as we will see in Section  \ref{sec_numerics}, computing these numbers 
is a rather simple task and it can be done very efficiently.

As the next corollary shows, the L\'evy measure of a Meromorphic process $X$ can be easily reconstructed from zeros and poles of $q-\psi(z)$. 
This has the spirit of Vigon's theory of philanthropy (cf. \cite{vigt}) in that we construct the L\'evy measure from the Wiener-Hopf factors. 

\begin{corollary}\label{corollary_Levy_measure} 
Assume that $q>0$ and the Wiener-Hopf factors are given by (\ref{eqn_WH_factor}).  Then $X$ is Meromorphic and it's L\'evy measure is given by 
(\ref{def_pi}), where $\{\rho_n\}_{n\ge 1}$ and $\{\hat \rho_n\}_{n\ge 1}$ are poles of the Wiener-Hopf factors $\phiqp(-\i z)$ and $\phiqm(\i z)$ and coefficients
$a_n$ and $\hat a_n$ are given by  
\beq\label{formula_an_hat_an}
a_n= \coeffb_n(\zeta, \rho) \frac{q}{\phiqm(-\i \rho_n)} , \;\;\; 
\hat a_n= \coeffb_n(\hat \zeta,\hat \rho)  \frac{q}{\phiqp(\i \hat \rho_n)}.
\eeq
\end{corollary}
\begin{proof}
The fact that $X$ is Meromorphic was already established in part (iii) of Theorem \ref{thm_equivalent}. 
Assume that the density of the L\'evy meausre $\pi(x)$ is given by (\ref{def_pi}), then as we have established in the proof of Theorem  \ref{thm_equivalent}, the 
Laplace exponent $\psi(z)$ can be expressed in the form (\ref{eq_psi}). Comparing this equation
with (\ref{WH_fact_1}) we conclude that the exponents $\rho_n$ and $\hat \rho_n$ in (\ref{def_pi}) must coincide with the poles 
of $\psi(z)$. 

From (\ref{eq_psi}) it also follows that $a_n \rho_n=- {\textrm {Res}} (\psi(z) : z=\rho_n)$. Using the Wiener-Hopf factorization
\beq\label{WH2}
 q-\psi(z)=\frac{q}{\phiqp(-\i z)\phiqm(-\i z)},
\eeq
and the fact that $\phiqm(-\i z)$ is non-zero at $\rho_n$ we see that
\beqq
a_n \rho_n = {\textrm {Res}} (\phiqp(-\i z)^{-1} : z=\rho_n) \times \frac{q}{\phiqm(-\i \rho_n)}. 
\eeqq
In order to finish the proof we only need to check that ${\textrm {Res}} (\phiqp(-\i z)^{-1} : z=\rho_n)=\rho_n \coeffb_n(\zeta, \rho) $ which
follows from (\ref{part_fraction2}). 
\end{proof}

 Note that the coefficients $a_n$ and $\hat a_n$ given in (\ref{formula_an_hat_an}) 
 (which define the L\'evy measure via (\ref{def_pi}))  depend on $q$ on the right hand side but not on the left hand side. It is therefore tempting to take limits as $q\downarrow 0$ on the right hand side. This is not as straightforward as it seems as the limits of $a_n$ and $\hat{a}_n$ are not easy to compute. It would be more straightforward to start with the Wiener-Hopf factorization for the case that $q = 0$ and then perform the same 
analysis as in Corollary \ref{corollary_Levy_measure}. The next corollary provides us with the aforementioned Wiener-Hopf factorization. We recall
that $\zeta_n=\zeta_n(q)$ (resp. $\hat \zeta_n=\hat \zeta_n(q)$)  denote the non-negative solutions to equation $\psi(z)=q$  (resp. $\psi(-z)=q$).


\begin{corollary}\label{corollay_q=0} 
${}$
\\
\begin{itemize}
\item[(i)]
Assume that $\e [X_1]>0$. As $q \to 0^+$ we have $q/\zeta_1(q) \to \e [X_1]$ and for $n\ge 1$ 
\beqq
\zeta_{n+1}(q) \to \zeta_{n+1}(0)\ne 0, \;\;\;  \hat\zeta_n(q) \to \hat \zeta_{n}(0)\ne 0.
\eeqq
We have the Wiener-Hopf factorization
$-\psi(z)=\kappa(0,-z)\hat \kappa(0,z)$, where
 \beq\label{eq_kappa_0}
 \kappa(0,z)&=&\lim\limits_{q\to 0^+} \frac{q}{\phiqp(\i z)} = z \e [X_1]  \prod\limits_{n\ge 1} \frac{1+\frac{z}{\zeta_{n+1}(0)}}{1+\frac{z}{\rho_n}}, \\ \nonumber
  \hat \kappa(0,z)&=&\lim\limits_{q\to 0^+} \frac{1}{\phiqm(-\i z)}=\frac{1}{\phi_0^{\minus}(-\i z)}=\prod\limits_{n\ge 1} \frac{1+\frac{z}{\hat \zeta_{n}(0)}} {1+\frac{z}{\hat \rho_n}}.
 \eeq 
 \item[(ii)]
Assume that $\e [X_1]=0$. As $q \to 0^+$ we have $\sqrt{q}/\zeta_1(q) \to \sqrt{{\textnormal{Var}}(X_1)/2}$, 
 $\sqrt{q}/\hat \zeta_1(q) \to \sqrt{{\textnormal{Var}}(X_1)/2}$ and for $n\ge 2$ 
\beqq
\zeta_{n}(q) \to \zeta_{n}(0)\ne 0, \;\;\; \hat\zeta_n(q) \to \hat \zeta_{n}(0)\ne 0.
\eeqq
We have the Wiener-Hopf factorization
$-\psi(z)=\kappa(0,-z)\hat \kappa(0,z)$, where
 \beq\label{eq_kappa_0_pt2}
 \kappa(0,z)&=&\lim\limits_{q\to 0^+} \frac{\sqrt{q}}{\phiqp(\i z)} =z \sqrt{\frac{{\textnormal{Var}}(X_1)}{2}}  \prod\limits_{n\ge 1} \frac{1+\frac{z}{\zeta_{n+1}(0)}}{1+\frac{z}{\rho_n}}, \\ \nonumber
  \hat \kappa(0,z)&=&\lim\limits_{q\to 0^+} \frac{\sqrt{q}}{\phiqm(-\i z)}=
  z \sqrt{\frac{{\textnormal{Var}}(X_1)}{2}}  \prod\limits_{n\ge 1} \frac{1+\frac{z}{\hat \zeta_{n+1}(0)}} {1+\frac{z}{\hat \rho_n}}.
 \eeq 
\end{itemize}
\end{corollary}
\begin{proof}
Let us prove part (i). Since $\e [X_1]>0$ we know that $\xsup_{\ee(q)} \to +\infty$ and $\xinf_{\ee(q)} \to \xinf_{\infty}$ as $q\to 0^+$, and also that
$\xinf_{\infty}<\infty$  with  probability one. Therefore, if $z>0$  the Wiener-Hopf factor $\phiqp(\i z)=\e [\exp(-z \xsup_{\ee(q)})]$ must converge to $0$ as $q\to 0^+$
while $\phiqm(-\i z)=\e [\exp(z \xinf_{\ee(q)})]$ must converge to $\e [\exp(z \xinf_{\infty})]$ as $q\to 0^+$. Since all roots 
$\zeta_n(q)$ and $\hat \zeta_n(q)$ have nonzero limits as $q\to 0^+$ if $n\ge 2$ (this is true due to the interlacing condition (\ref{interlacing_prop}) and the fact that $\psi(z)$ is a meromorphic
function) we conclude that $\zeta_1(q)$ must go to zero while $\hat \zeta_1(q) \to \hat\zeta_1(0) \ne 0$. The function $\psi(z)$ is analytic in the neighbourhood 
or $z=0$, and it's derivative $\psi'(0)=\e [X_1] \ne 0$, thus using the Implicit Function Theorem 
we conclude that $\zeta_1(q)$ is an analytic function of $q$ in some neighbourhood  of $q=0$, and taking derivative with respect to $q$ of the equation $\psi(z)=q$ we find
that
 \beqq
 \frac{\d}{\d q} \zeta_1(q) = \frac{1 }{\psi'(\zeta_1(q))}.
 \eeqq
Since $\psi'(0)= \e [X_1]$ we conclude that $q/\zeta_1(q) \to \e [X_1]$ as $q\to 0^+$. Using this fact and the  formulae in (\ref{eqn_WH_factor}) for the Wiener-Hopf factors we obtain expressions  (\ref{eq_kappa_0}). 

The proof of part (ii) is very similar. We use the fact that $\psi'(0)=\e [X_1]=0$ and $\psi''(0)={\textnormal{Var}}(X_1)$ to conclude that
$\psi(z)=\frac12 {\textnormal{Var}}(X_1) z^2 + O(z^3)$ as $z\to 0$. This implies that when $q$ is small and positive, the equation $\psi(z)=q$ has two solutions in the neighbourhood of zero,  $\zeta_1=\sqrt{2q/{\textnormal{Var}}(X_1)}+o(\sqrt{q})$ and 
$-\hat\zeta_1=-\sqrt{2q/{\textnormal{Var}}(X_1)}+o(\sqrt{q})$. Using this information, the uniquenes of the Wiener-Hopf factorization and taking the limit as $q\to 0^+$ in (\ref{WH_fact_1})  we obtain (\ref{eq_kappa_0_pt2}).
\end{proof}


\section{Examples of Meromorphic processes}\label{examples}

There are four particular families of L\'evy processes which have featured in recent literature, all of which have proved to have relevance to a number of applied probability models, all of which have exhibited some degree of mathematical tractability in the context that they have occurred. These processes are as follows.

\begin{itemize}
\item {\bf Hyper-exponential L\'evy processes:} The elementary but classical Kou model, \cite{Kou}, consists of a linear Brownian motion plus a compound Poisson process with two-sided exponentially distributed jumps. The natural generalization of this model (which therefore includes the Kou model) is the hyper-exponential L\'evy process for which the exponentially distributed jumps are replaced by hyper-exponentially distributed jumps. That is to say, the density of the 
L\'evy measure is written
\[
\pi(x) =
  \ind_{\{x>0\}}\sum_{i=1}^{N} a_i \rho_i e^{-\rho_i x} 
+
  \ind_{\{x<0\}}\sum_{i=1}^{\hat N} \hat a_i \hat \rho_i e^{\hat \rho_i x},
\]
where $a_i$, $\hat a_i$, $\rho_i$ and $\hat \rho_i$ are positive numbers and $N$, $\hat N$ are positive integers. One can verify that the Laplace exponent is a rational function
of the form (\ref{eq_psi}), where we have finite sums instead of infinite series, and that hyper-exponential L\'evy processes have finite activity jumps and paths of bounded variation unless $\sigma^2>0$.

Strictly speaking, the Hyper-exponential L\'evy processes do not belong to the class of Meromorphic processes, as in this case 
the Laplace exponent $\psi(z)$ has only finite number of poles $\rho_n$, $-\hat \rho_n$. However, all the results presented in this paper are 
still correct when we have a process with positive and/or negative hyper-exponential jumps, provided that the results are interpreted in the correct way, see Remark \ref{hyp_remark} below.

\item {\bf $\beta$-class L\'evy processes:} The $\beta$-class of L\'evy processes was introduced by Kuznetsov \cite{Kuz-beta}. The Laplace exponent is given by
\begin{eqnarray*}
\psi(z) &=&  \frac{1}{2}\sigma^2 z^2 + az+ \frac{c_1}{\beta_1} \left\{ {\textnormal{B}}( \alpha_1 - \frac{z}{\beta_1}, 1- \lambda_1)-
{\textnormal{B}}(\alpha_1, 1-\lambda_1)\right\}\\
&& + \frac{c_2}{\beta_2}\left\{{\textnormal{B}}( \alpha_2 + \frac{ z}{\beta_2}, 1- \lambda_2)-{\textnormal{B}}(\alpha_2, 1-\lambda_2) \right\},
\end{eqnarray*}
where ${\textnormal{B}}(x,y)=\Gamma(x)\Gamma(y)/\Gamma(x+y)$ is the Beta function, with parameter range $a\in\mathbb{R}$, $\sigma\ge0$, $c_i, \alpha_i, \beta_i >0$ and $\lambda_i\in(0,3)\setminus\{1,2\}$. The corresponding  L\'evy measure $\Pi(\d x)$  has density
\[
\pi(x) = \ind_{\{x>0\}} c_1\frac{e^{-\alpha_1 \beta_1 x}}{(1- e^{-\beta_1 x})^{\lambda_1}}
+ \ind_{\{x<0\}}c_2 \frac{e^{\alpha_2\beta_2 x}}{(1- e^{\beta_2 x})^{\lambda_2}}.
\]
With the help of binomial series one can check that $\pi(x)$ is of the form (\ref{def_pi}), thus $\beta$-processes are Meromorphic. 
 
The large number of parameters allows one to choose L\'evy processes within the $\beta$-class that have paths that are both of unbounded variation (when at least one of the conditions $\sigma\neq 0$, $\lambda_1\in(2,3)$ or $\lambda_2\in(2,3)$ holds) and bounded variation (when all of the conditions $\sigma=0$, $\lambda_1\in(0,2)$ and $\lambda_2\in(0,2)$ hold) as well as having infinite and finite activity in the jump component (accordingly as both $\lambda_1,\lambda_2\in(1,3)$ or $\lambda_1,\lambda_2\notin(1,3)$). The $\beta$-class of L\'evy processes includes another recently introduced family of L\'evy processes known as Lamperti-stable processes, cf. \cite{CaCha, Caballero2008, Chaumont2009}.

\item {\bf $\theta$-class L\'evy processes:}

 The $\theta$-class of L\'evy processes was also introduced by Kuznetsov \cite{Kuz-theta} as a family of processes
with Gaussian component and two-sided jumps, characterized by the density of the L\'evy measure 
\beqq
\pi(x)=\ind_{\{x>0\}} c_1\beta_1 e^{-\alpha_1 x}\Theta_{k}(x\beta_1) + \ind_{\{x<0\}} c_2 \beta_2 e^{\alpha_2 x}\Theta_{k}(-x\beta_2),
\eeqq
where $\Theta_{k}(x)$ is the $k$-th order (fractional) derivative of the theta function $\theta_3(0,e^{-x})$:
\beqq
\Theta_{k}(x)=\frac{\d ^k}{\d x^k} \theta_3(0,e^{-x}) = \ind_{\{k=0\}}+2\sum\limits_{n\ge 1} n^{2k} e^{-n^2 x}, \;\;\; x>0.
\eeqq
The parameter $\chi=k+1/2$ corresponds to the exponent of the singularity of $\pi(x)$ at $x=0$, and when $\chi \in \{1/2,3/2,5/2\}$ ($\chi \in \{1,2\}$) the Laplace exponent of $X$ is given in terms of trigonometric (digamma) functions. Again, due to the fact that the density of 
the L\'evy measure is of the form (\ref{def_pi}) we conclude that $\theta$-processes are Meromorphic.

\item {\bf Hypergeometric L\'evy processes:} These processes were introduced in Kyprianou et al. \cite{KyPaRi, KyPavS} as an example of how to use Vigon's theory of philanthropy. Their Laplace exponent is given by 
\beq\label{def_hypergeometric}
\psi(z) = \frac{1}{2}\sigma^2 z^2 +az+  \Phi_1(-z)\Phi_2(z), 
\eeq
where $a,\sigma \in\mathbb{R}$ and $\Phi_1,\Phi_2$ are the Laplace exponents of two subordinators from the $\beta$-class of L\'evy processes. Note that such Laplace exponents necessarily take the form 
\begin{equation}
\Phi(\theta)
 =\mathtt{k}+\delta\theta +\frac{c}{\beta} \big\{ 
 {\textnormal{B}}\left(1-\alpha+\gamma, - \gamma \right) - {\textnormal{B}}\left( 1-\alpha+\gamma+\theta/\beta , - \gamma \right)\big\}.\label{kuz}
\end{equation}

Compared with the previous examples, it is less obvious that hypergeometric processes are Meromorphic, 
however it is not hard to establish this result with the help of Theorem \ref{thm_equivalent}. We know that $\beta$-processes are meromorphic, thus the process $\tilde X$ with the Laplace exponent 
$\tilde \psi(z)= \Phi_1(-z)\Phi_2(z)$ is Meromorphic (this is due to part (v) of Theorem \ref{thm_equivalent}). Finally, the hypergeometric process defined by the Laplace exponent via (\ref{def_hypergeometric}) and the process $\tilde X$ have the same L\'evy measure, therefore $X$ is also Meromorphic.
\end{itemize}

\begin{remark}\label{hyp_remark} Within the scope of Definition \ref{maindef},
when the process $X_t$ has hyper-exponential positive jumps (a mixture of $N$ exponentials), 
 we have only a finite sequence $\{\rho_n\}_{n=1,2,..,N}$ and either $N$ (or $N+1$) negative roots $\zeta_N$. 
All the formulas presented in this paper are still valid in this case if we adopt notation
 $\rho_k = \infty$ for $k > N$ and $\zeta_k=\infty $ for $k>N$ (or $k>N+1$). For example, expression (\ref{eqn_WH_factor}) for the Wiener-Hopf factor becomes
a finite product
 \beqq
\phiqp(\i z)=\prod\limits_{n=1}^N \frac{1+\frac{z}{\rho_n}}{1+\frac{z}{\zeta_n}}, \;\;\; 
\textrm{ or } \;\;\; \phiqp(\i z)= \left[ \prod\limits_{n=1}^N \frac{1+\frac{z}{\rho_n}}{1+\frac{z}{\zeta_n}} \right] \frac{1}{1+\frac{z}{\zeta_{N+1}}},
 \eeqq 
depending on whether we have $N$ or $N+1$ negative roots $\zeta_n$. The same remark holds true when we have hyper-exponential negative jumps.
\end{remark}

\begin{remark}
When the process $X$ has hyper-exponential positive jumps, calculation of coefficients $\coeffa(\rho,\zeta)$ and $\coeffb(\rho,\zeta)$ is much easier, since the corresponding expressions in 
(\ref{formula_cn_a_b}) and (\ref{formula_cn_b_a}) are just finite products. Moreover, calculation of the coefficients $\coeffa(\hat\rho,\hat\zeta)$
and $\coeffb(\hat\rho,\hat\zeta)$ is also simplified considerably even if the sequence $\{\hat \rho_n\}_{n\geq 1}$ is infinite, since we can use Wiener-Hopf factorization (\ref{WH2}) and the fact that $-\hat\zeta_n$ and $-\hat \rho_n$ are simple roots (poles) of function $\psi(z)-q$. For example we can compute
coefficients $\coeffa_n(\hat\rho,\hat\zeta)$ as follows:
 \beqq
\coeffa_n(\hat\rho,\hat\zeta)&=&\frac{1}{\hat \zeta_n} {\textrm {Res}}\left(\phiqm(-\i z): z=- \hat \zeta_n \right)\\
&=&
\frac{1}{\hat \zeta_n}{\textrm {Res}}\left( \frac{q}{(q-\psi(z))\phiqp(- \i z)} : z=- \hat \zeta_n\right)=
- \frac{q}{\hat \zeta_n \psi'(-\hat \zeta_n) \phiqp(\i \hat \zeta_n)},
 \eeqq
and this last expression is more convenient than (\ref{formula_cn_a_b}), since usually we know $\psi'(z)$ explicitly and $\phiqp(z)$ is just a rational function. 
\end{remark}

\section{One-sided exit problem}\label{one-sided}

Doney and Kyprianou \cite{Doney} have given a detailed characterization of the one-sided exit problem through the so-called quintuple law. The latter can be easily integrated out to give the following discounted triple law which is also known in the actuarial mathematics literature as the Gerber-Shiu measure (cf. \cite{biff}). The following result holds for general L\'evy processes.

\begin{lemma} Fix $c>0$. Define
\[
\tau^+_c = \inf\{t>0: X_t >c\}.
\] 
For all $q,y, z>0$ and  $u\in[0,c \vee z]$ we have 
\beqq
&&\mathbb{E}\left[ e^{-q\tau^+_c} \ind( X_{\tau^+_c}-c\in{\rm d}y \; ; \;  c - X_{\tau^+_c-}\in{\rm d}z \; ;\;  c - \overline{X}_{\tau^+_c-} \in {\rm d}u)   \right]
\\
&& \hspace{3cm}=\frac{1}{q} \p(\xsup_{\ee(q)} \in c  - \d u) \p(-\xinf_{\ee(q)} \in \d z - u) \Pi({\rm d}y + z).
\eeqq
\end{lemma}

Taking account of formulas (\ref{dist_sup_X}) this gives us the following immediate corollary for the $M$-class of L\'evy processes.

\begin{corollary}
Fix $c>0$. For all $q,y, z>0$ and $u\in[0,c \vee z]$  we have 
\beqq
&&\mathbb{E}\left[ e^{-q\tau^+_c} \ind( X_{\tau^+_c}-c\in{\rm d}y \; ; \;  c - X_{\tau^+_c-}\in{\rm d}z \; ;\;  c - \overline{X}_{\tau^+_c-} \in {\rm d}u)   \right]
 \\
&& \hspace{3cm}= \frac{1}{q} \left[ \bar\coeffa(\rho,\zeta)^T \times \bar\vv(\zeta,c-u) \right] 
\left[ \bar\coeffa(\hat \rho,\hat \zeta)^T \times \bar\vv(\hat \zeta,z - u)\right]\Pi({\rm d}y + z)\d z \d u.
\eeqq
\end{corollary}

Often, one is only interested in the discounted overshoot distribution. In principle this can be obtained from the above formula by integrating out the variables $z$ and $u$. However, it turns out to be more straightforward to prove this directly, particularly as otherwise it would require us to specify in more detail the L\'evy measure in terms of poles and roots for the $M$-class (cf. Corollary \ref{corollary_Levy_measure}). The following result does precisely this, but in addition, it also gives us the probability of creeping.

\begin{theorem}\label{thm_one_sided_exit} 
 Define a matrix $\aa=\{a_{i,j}\}_{i,j\ge 0}$ as
\beq\label{def_A}
 a_{i,j}=
 \begin{cases}
  0  \;\;\; & {\textnormal {if  }}  i=0, \; j\ge 0 \\
  \coeffa_i(\rho,\zeta) \coeffb_0(\zeta,\rho) \;\;\; & {\textnormal {if  }}   i \ge 1, \; j=0 \\
 \displaystyle \frac{\coeffa_i(\rho,\zeta) \coeffb_j(\zeta,\rho) }{\rho_j-\zeta_i}  \;\;\; &{\textnormal {if  }} i\ge 1, \; j \ge 1 
 \end{cases}
\eeq
Then for $c>0$ and $y\geq 0$ we have 
 \beq\label{eqn_one_sided_exit}
   \e \left[ e^{-q \tau_c^+} \ind \left(X_{\tau_c^+} -c\in \d y\right) \right]&=& 
\bar \vv(\zeta,c)^T\times \aa \times \bar \vv(\rho,y) \d y.
 \eeq
\end{theorem}

\begin{proof}
We start with the formula from Lemma 1 in \cite{Alili2005}
\beq\label{Al_Kyp_identity}
 \e \left[ e^{-q \tau_c^+ -z (X_{\tau_c^+} -c) }  \right]=
\frac{\e\left[ e^{-z (\overline X_{\ee(q)}-c)} \ind ( \xsup_{\ee(q)}>c)\right]}{\e\left[ e^{-z \overline X_{\ee(q)}} \right]}.
\eeq
Using formula (\ref{dist_sup_X}) we find that
 \beqq
 \e\left[ e^{-z (\overline X_{\ee(q)}-c)} \ind ( \xsup_{\ee(q)}>c)\right]=
\sum\limits_{i\ge 1} \coeffa_i(\rho,\zeta) \frac{\zeta_i }{z+\zeta_i} e^{-\zeta_i c}.
\eeqq
Next, using  (\ref{eqn_WH_factor}) and (\ref{part_fraction2}) after some algebraic manipulations we find that 
 \beqq
  \frac{1}{(z+\zeta_i)\e\left[ e^{-z \overline X_{\ee(q)}} \right]} =
  \frac{1}{(z+\zeta_i)} \prod\limits_{n\ge 1}  \frac{1+\frac{z}{\zeta_n}}{1+\frac{z}{\rho_n}} &=&
  \coeffb_0(\zeta,\rho)+
\sum\limits_{j\ge 1} \frac{\coeffb_j(\zeta,\rho)}{\rho_j-\zeta_i} \frac{\rho_j}{\rho_j+z}.
 \eeqq
Combining the above three equations we find
\beqq
 \e \left[ e^{-q \tau_c^+ -z (X_{\tau_c^+} -c)}  \right]=\sum\limits_{i\ge 1} \coeffa_i(\rho,\zeta) \coeffb_0(\zeta,\rho) \zeta_i  e^{-\zeta_i c} +
\sum\limits_{i\ge 1}\sum\limits_{j\ge 1}  \frac{\coeffa_i(\rho,\zeta) \coeffb_j(\zeta,\rho) }{\rho_j-\zeta_i}  \zeta_i  e^{-\zeta_i c} \frac{\rho_j}{\rho_j+z}
\eeqq
which allows for a straightforward inversion in $z$, thereby completing the proof.
\end{proof}

\section{Entrance/exit problems for a finite interval}\label{sec_dble_exit}

The two sided exit problem is a long standing problem of interest from the theory of L\'evy processes and there are very few cases where explicit identities have been be extracted for the (discounted) overshoot distribution on either side of the interval. Classically the only examples with discounting which have been analytically tractable are those of jump diffusion processes with (hyper)-exponentially distributed jumps (see for example \cite{KV07, WZ07}) and, up to knowing the so-called scale function, spectrally negative processes, \cite{Kyprianou}. Otherwise, the only other examples, without discounting, are those of stable processes, \cite{Rogozin}, and Lamperti-stable processes, \cite{Chaumont2009}.

Recall that
\[
\tau^+_a = \inf\{t>0 : X_t >a\}, \;\;\; \tau^-_0 = \inf\{t>0 : X_t <0\}.
\]
For $f \in \ll_{\infty}(\r)$, the space of positive, measurable and uniformly bounded functions on $\mathbb{R}$, and $x\in\mathbb{R}$, define the following operators
 \beq\label{def_G}
 (\gg f)(x)&=&\ind (0\leq  x \leq a )  \e_x \left[ e^{-q \tau_a^+} f(X_{\tau_a^+}) \ind \left(\tau_a^+ < \tau_0^- \right)\right], \\ \nonumber
 (\hat \gg f)(x)&=&\ind (0 \leq  x \leq a)  \e_x \left[ e^{-q \tau_0^-} f(X_{\tau_0^-}) \ind \left(\tau_0^- < \tau_a^+ \right)\right] ,
\eeq
and for any subset $I \in \r$  define
\beq\label{def_P}
 (\pp_{I} f)(x)&=& \ind ( x \in I) \e_x \left[ e^{-q \tau_a^+}  f(X_{\tau_a^+})\right], \\ \nonumber
(\hat \pp_{I} f)(x)&=&\ind ( x \in I) \e_x \left[ e^{-q \tau_0^-}  f(X_{\tau_0^-}) \right].
 \eeq
These are bounded operators from $\ll_{\infty}(\r)$ into $\ll_{\infty}(\r)$.
Note also that if $X_0<0$ then  $\tau_0^-=0$ and similarly if $X_0>a$ then $\tau_a^+=0$. For example, when $x>a$, $\e_x \left[ e^{-q \tau_a^+}  f(X_{\tau_a^+})\right] = f(x)$ and when $x<0$, $\e_x \left[ e^{-q \tau_0^-}  f(X_{\tau_0^-}) \right] = f(x)$. 

Following similar reasoning to Rogozin \cite{Rogozin} (see also \cite{KK05}), an application of the 
Markov property tells us that for all $x\in [0,a]$ we have
 \beqq
  \e_x\left[ e^{-q \tau_a^+}  f(X_{\tau_a^+})\right]&=& \e_x \left[ e^{-q \tau_a^+} f(X_{\tau_a^+}) \ind \left(\tau_a^+ < \tau_0^- \right)\right]\\
 &&+  \e_x \left[ e^{-q \tau_0^-} \ind \left(\tau_0^- < \tau_a^+ \right)
\e_{X_{\tau_0^-}} \left[ e^{-q \tau_a^+}  f(X_{\tau_a^+})\right]
\right].
 \eeqq
Thus we have the following operator identities:
\beq\label{eq_operator}
 \begin{cases}
\pp_{[0,a]}&=\gg+\hat \gg \pp_{(-\infty,0]}, \\
\hat \pp_{[0,a]}&=\hat \gg + \gg \hat \pp_{[a,\infty)}.
\end{cases}
\eeq 
Consider a space $\hh$ of bounded linear operators on $\ll_{\infty}(\r)$ with the norm
\beqq
 ||\gg||=\sup\limits_{ ||f||_{\infty}\le 1 }  || \gg f ||_{\infty} .
\eeqq

In the next theorem, which holds for general L\'evy processes, we give a series representation for $\mathcal{G}$ in terms of $\hat \pp_{[0,a]},\, \pp_{(-\infty,0]},\, \hat \pp_{[a,\infty)},\,  \pp_{(-\infty,0]}$ and a similar one for $\hat{\mathcal{G}}$.
In the case of stable processes and general L\'evy processes respectively,  Rogozin \cite{Rogozin} and Kadankov and Kadankova \cite{KK05} have shown similar series representations. The presentation we give above only differs in that it takes the form of linear operators. This turns out to be more convenient for the particular application to the $M$-class of processes.

\begin{theorem}\label{solution_to_operator_eqns}
Assume $q>0$. A pair of operators  $(\gg,\hat \gg)$ given by (\ref{def_G}) is the unique solution in $\hh$ to the system of equations  (\ref{eq_operator}). 
Moreover, this solution can also be represented in series form
\begin{equation}
\gg =\pp_{[0,a]}-\hat \pp_{[0,a]} \pp_{(-\infty,0]}+\pp_{[0,a]} \hat \pp_{[a,\infty)} \pp_{(-\infty,0]}-\hat \pp_{[0,a]} \pp_{(-\infty,0]}\hat \pp_{[a,\infty)} \pp_{(-\infty,0]}+\dots 
\label{inf_srs_sltn_to_G_hat_G}
\end{equation}
and
\[
\hat \gg =\hat \pp_{[0,a]}- \pp_{[0,a]} \hat \pp_{[a,\infty)}+ \hat \pp_{[0,a]} \pp_{(-\infty,0]} \hat \pp_{[a,\infty)} -
\pp_{[0,a]} \hat \pp_{[a,\infty)} \pp_{(-\infty,0]} \hat \pp_{[a,\infty)} +\dots .
\]
where, in both cases, convergence is exponential in the following sense. There exist $\eta\in(0,1)$ and a constant $C>0$ such that if $\mathcal{S}_n$ is the sum of the first $n$ terms in the series describing $\gg$ then 
$
||\gg - \mathcal{S}_n||\leq C\eta^{n+1},
$
with a similar statement holding for $\hat \gg$.

\end{theorem}
\begin{proof}  It suffices to prove the result for $\gg$, the proof for $\hat \gg$ follows by duality.
For  uniqueness, we note first that the norm of operators $\pp_{(-\infty,0]}$ is stricly less than one:
\beqq
|| \pp_{(-\infty,0]}|| = \sup\limits_{ ||f||_{\infty}\le 1 } \left[ \sup\limits_{x\in \r} \ind ( x\le 0 )  \e_x \left[ e^{-q \tau_a^+}  f(X_{\tau_a^+})\right] \right] \le
  \sup\limits_{x\in \r} \ind ( x\le 0 )  \e_x \left[ e^{-q \tau_a^+} \right]=  \e_0 \left[ e^{-q \tau_a^+} \right]<1
\eeqq
and similarly we find  $||\hat \pp_{[a,\infty)}||<1$. Uniqueness also follows from the latter fact. Indeed if we assume that we have another pair of solutions 
$( G, \hat G)$, then using the system of equations  (\ref{eq_operator}) we find that the difference $\gg-G$ must satisfy the equation
\beqq
 \gg - G = (\gg - G)  \hat \pp_{[a,\infty)} \pp_{(-\infty,0]}.
\eeqq
Thus we find 
\beqq
|| \gg - G|| \le || \gg - G|| \times ||  \pp_{[a,\infty)} || \times || \pp_{(-\infty,0]} ||
\eeqq
and since $||  \pp_{[a,\infty)} || \vee || \pp_{(-\infty,0]} ||<1$ we conclude that $||\gg - G||=0$ and $G=\gg$.
To establish the series representation (\ref{inf_srs_sltn_to_G_hat_G}) we use (\ref{eq_operator}) to find 
\beqq
\gg =\pp_{[0,a]}-\hat \pp_{[0,a]} \pp_{(-\infty,0]}+\gg \hat \pp_{[a,\infty)} \pp_{(-\infty,0]}.
\eeqq
By iterating this equation we obtain series representation (\ref{inf_srs_sltn_to_G_hat_G}), which converges at an exponential rate as described in the statement of the theorem with 
$\eta = || \hat \pp_{[a,\infty)} || \vee || \pp_{(-\infty,0]} ||\in(0,1)$.
\end{proof}

The next theorem converts the above general setting to the specific setting of $M$-processes.

%

\begin{theorem}{\bf (First exit from  a finite interval)}\label{thm_two_sided} 
Let $a>0$ and define a matrix $\bb=\bb(\hat \rho, \zeta,a)=\{b_{i,j}\}_{i,j\ge 0}$  with
\beq\label{def_B}
 b_{i,j}=
 \begin{cases}
 \displaystyle  \zeta_j e^{-a \zeta_j} \;\;\; & {\textnormal {if  }} i=0, \; j\ge 1 \\ 
  0\;\;\; & {\textnormal {if  }}  i \ge 0, \; j=0 \\ 
 \displaystyle \frac{\hat \rho_i \zeta_j}{\hat \rho_i+\zeta_j}  e^{- a \zeta_j}  \;\;\; &{\textnormal {if  }} i\ge 1, \; j \ge 1 ,
 \end{cases}
\eeq
and  similarly $\hat \bb=\bb(\rho, \hat \zeta,a)$.There exist matrices $\cc_1$, $\cc_2$ and $\hat \cc_1$, $\hat \cc_2$ such that for $x \in (0,a)$  we have
 \beq\label{eq_dble_exit}\nonumber 
  \e_x \left[ e^{-q \tau_a^+} \ind \left(X_{\tau_a^+} \in \d y \; ; \; \tau_a^+<\tau_0^- \right)\right]&=&
 \left[ \bar \vv(\zeta,a-x)^T \times \cc_1 + \bar \vv(\hat \zeta,x)^T \times \cc_2 \right]  \times \bar \vv(\rho, y-a) \d y \\ 
\e_x \left[ e^{-q \tau_0^-} \ind \left(X_{\tau_0^-} \in \d y \; ; \; \tau_0^-<\tau_a^+ \right)\right]&=&
 \left[  \bar \vv(\hat \zeta,x)^T \times \hat \cc_1 + \bar \vv(\zeta,a-x)^T \times \hat \cc_2 \right] \times \bar \vv(\hat \rho, -y) \d y
 \eeq
These matrices satisfy the following system of linear equations
 \beq\label{matrix_eqns}
\begin{cases}
 \cc_1&=\aa-\hat \cc_2 \bb \aa \\
 \hat \cc_2  &= -  \cc_1 \hat \bb \hat \aa 
\end{cases} \qquad
 \begin{cases}
 \hat \cc_1&=\hat \aa- \cc_2 \hat \bb \hat \aa \\
 \cc_2  &= -  \hat \cc_1 \bb \aa
\end{cases}
 \eeq
where $\aa$ and $\hat \aa$ are defined by (\ref{def_A}). The system of linear equations (\ref{matrix_eqns}) can be solved iteratively with exponential convergence with respect to the matrix norm $||\cdot||_\infty$,  where for any square matrix $\mathbf{M}$, 
 \beq\label{matrix_norm}
 ||{\mathbf M}||_{\infty}=\max\limits_{i\ge 0} \sum\limits_{j\ge 0} |M_{i,j}|.
 \eeq
\end{theorem}

Before proceeding to the proof of the above theorem, there is one technical lemma we must first address.

\begin{lemma}\label{lemma_matrix_norm}
 Let ${\mathbf H}=\bb \times \aa$. Then $||{\mathbf H}||_{\infty}<1$. 
\end{lemma}
\begin{proof}
Using definitions (\ref{def_A}) and (\ref{def_B}) of $\aa$ and $\bb$ we find that for $i,j \ge 1$
 \beqq
 h_{0,0}&=&\coeffb_0(\zeta, \rho)\sum\limits_{k \ge 1} \coeffa_k(\rho,\zeta)  \zeta_k e^{-a \zeta_k} \\
 h_{0,j}&=&\coeffb_j(\zeta, \rho) \sum\limits_{k \ge 1} \coeffa_k(\rho,\zeta) \frac{\zeta_k e^{-a \zeta_k}}{\rho_j-\zeta_k} \\
 h_{i,0}&=&\hat \rho_i \coeffb_0(\zeta, \rho)  \sum\limits_{k \ge 1} \coeffa_k(\rho,\zeta) \frac{\zeta_k e^{-a \zeta_k}}{\zeta_k+\hat \rho_i} \\
 h_{i,j}&=&\hat \rho_i \coeffb_j(\zeta, \rho) \sum\limits_{k \ge 1} \coeffa_k(\rho,\zeta) \frac{\zeta_k e^{-a \zeta_k}}{(\rho_j-\zeta_k)(\zeta_k+\hat \rho_i)}.
 \eeqq
Our first goal is to prove that $h_{i,j}$ are positive for all $i,j\ge 0$ (note that $h_{i,0}>0$ for all $i
\ge 0$). 
Let us consider a random variable
$\xi=\ee(\rho_j)+\xsup_{\ee(q)}$, where as usual $\ee(\rho_j)$ is an exponentially distributed random variable with rate $\rho_j$, independent of $X$ and $\ee(q)$. The Laplace transform of 
$\xi$ is given by
 \beqq
 \e \left[ e^{-z \xi}\right]=\e \left[ e^{-z \ee(\rho_j)}\right]\times \e \left[ e^{-z \xsup_{\ee(q)}}\right]=  
  \frac{\rho_j}{\rho_j+z}\prod\limits_{n\ge 1}  \frac{1+\frac{z}{\rho_n}}{1+\frac{z}{\zeta_n}} ,
 \eeqq
and using the partial fraction decomposition (\ref{part_fraction1}), after some algebraic manipulations we find that $\xi$ has a probability density function
 \beqq
 \frac{\d}{\d x}\p(\xi \le x) = \rho_j \sum\limits_{k\ge 1} \coeffa_k(\rho,\zeta) \frac{ \zeta_k e^{-x \zeta_k } }{\rho_j-\zeta_k}.
 \eeqq
This proves that $h_{0,j}$ are all positive. To prove that $h_{i,j}$ are positive one can use identity
\beqq
 \e\left[ e^{-\hat \rho_i (\xi -a) }\ind(\xi > a) \right]=
\rho_j \sum\limits_{k\ge 1} \coeffa_k(\rho,\zeta) \frac{\zeta_k e^{-a \zeta_k } }{(\rho_j-\zeta_k)(\zeta_k+\hat \rho_i)}.
\eeqq

Next we need to prove that $||{\mathbf H}||_{\infty}<1$. Define a row vector $\bar {\textnormal h}_i$ as the $i$-th row of the matrix ${\mathbf H}$.
Using Theorem \ref{thm_one_sided_exit} we check that
 \beqq
  \e \left[ e^{-q \tau_a^+} \ind(X_{\tau_a^+} -a \in \d y) | X_0=-\eta_i \right] = \bar {\textnormal h}_i \times \bar \vv(\rho, y)\d y,
 \eeqq
where $\eta_0\equiv 0$ and $\eta_i $ is independent and exponentially distributed with parameter $\hat\rho_i$. 
Thus we have
\beqq
 \sum\limits_{j\ge 0} h_{i,j}=\e \left[ e^{-q \tau_a^+} | X_0=-\eta_i \right]<\e \left[ e^{-q \tau_a^+} | X_0=0 \right]<1,
\eeqq
and using the already established fact that $h_{i,j}$ are positive we conclude that $||{\mathbf H}||_{\infty}<1$.
\end{proof}

{\it Proof of Theorem \ref{thm_two_sided}.}  The 
operator $\pp_I$ is an integral operator, with  kernel
\beqq
\ind ( x \in I) \e_x \left[ e^{-q \tau_a^+}  \ind (X_{\tau_a^+} \in \d y)\right]=
\ind ( x \in I, y \ge a )\bar \vv(\zeta,a-x)^T\times \aa \times \bar \vv(\rho,y-a) \d y,
\eeqq
(see Theorem \ref{thm_one_sided_exit}). We also have a similar formula for $\hat \pp_I$. Using the infinite series representation
 (\ref{inf_srs_sltn_to_G_hat_G}) we find that there exist some matrices of coefficients $\cc_1$, $\cc_2$ and $\hat \cc_1$, $\hat \cc_2$ so that
  integral kernels of operators $\gg$ and $\hat \gg$ can be represented in the form (\ref{eq_dble_exit}). 
Matrix equations   (\ref{matrix_eqns}) follow from operator 
identities (\ref{eq_operator}) and the fact that
\beqq
\int\limits_0^{\infty} \bar \vv(\hat \rho, z) \times  \bar\vv(\zeta,a+z)^T \d z = \bb.
\eeqq
Equations (\ref{matrix_eqns}) can be solved iteratively as follows: 
 \beqq
 \cc_1^{(n+1)}&=&\aa+ \cc_1^{(n)} \hat \bb \hat \aa  \bb \aa, \;\;\; \cc_1^{(0)}={\bf 0} \\
 \hat \cc_1^{(n+1)}&=&\hat \aa+ \hat \cc_1^{(n)}   \bb \aa \hat \bb \hat \aa, \;\;\; \hat\cc_1^{(0)}={\bf 0},
 \eeqq
 and the convergence  $\cc_1^{(n)} \to \cc_1$ and $\hat \cc_1^{(n)} \to \hat \cc_1$ is exponential because of Lemma \ref{lemma_matrix_norm}. 
 Once we find $\cc_1$ and $\hat \cc_1$ we find $ \cc_2 = -  \hat \cc_1 \bb \aa$  and $ \hat \cc_2  = -  \cc_1 \hat \bb \hat \aa $.
\qed

\begin{remark}
Assume that within the $M$-class the positive (negative) jumps come from a mixture of $N<\infty$ ($\hat N<\infty$) exponential distributions. Then  matrices $\bb$ and $\hat \bb$ 
have size $(\hat N+1) \times (N+1)$ and $(N+1) \times (\hat N+1)$, while matrices $\cc_1$ and $\hat \cc_1$ have sizes $(N+1) \times (N+1)$  and $(\hat N+1) \times  (\hat N+1)$ and can be found by the relations
 \beqq
 \cc_1=\aa \left( \ii-\hat \bb \hat \aa \bb \aa \right)^{-1}, \;\;\; \hat \cc_1=\hat \aa \left( \ii-\bb \aa \hat \bb \hat \aa \right)^{-1}.
\eeqq
\end{remark}

\bigskip

The idea of writing down a pair of simultaneous equations (\ref{def_G}) also works when considering the problem of first entrance {\it into} an interval. This problem has been considered earlier by \cite{Rogozin} for the setting of stable processes and \cite{KV07} for two sided L\'evy processes whose upwards jumps are exponentially distributed.

\begin{theorem}[First entrance into a finite interval]\label{thm_interval_entrance} 
Let $a>0$ and define $\tau=\inf\{ t\ge 0: X_t \in (0,a)\}$. Define a matrix ${\bf M}=\bb(\hat \zeta,\rho,a)^T$ 
and similarly $\hat {\bf M}=\bb(\zeta,\hat \rho,a)^T$ (see (\ref{def_B})).
 There exist matrices ${\bf N}_1$, ${\bf N}_2$ and $\hat {\bf N}_1$, $\hat {\bf N}_2$ such that for  $y \in (0,a)$  we have
 \beq\label{eq_dble_entrance}
  \e_{x} \left[ e^{-q \tau} \ind \left(X_{\tau} \in \d y \right)\right]=
 \begin{cases}
 \bar \vv(\zeta, -x)^T \times  \left[ {\bf N}_1 \times \bar \vv(\rho,y) + {\bf N}_2 \times \bar \vv(\hat \rho,a-y) \right]  \times\d y,   &x \leq 0  \vspace{0.1cm} \\
 \bar  \vv(\hat \zeta, x-a)^T \times  \left[ \hat {\bf N}_1 \times \bar \vv(\hat \rho,a-y) + \hat {\bf N}_2 \times \bar \vv(\rho,y) \right]  \times\d y, &x \geq a.
\end{cases}
 \eeq
These matrices satisfy the following system of linear equations
 \beq\label{matrix_eqns2}
\begin{cases}
 {\bf N}_1&=\aa+\aa {\bf M} \hat {\bf N}_2 \\
 \hat {\bf N}_2  &=   \hat \aa \hat {\bf M}  {\bf N}_1
\end{cases} \qquad
 \begin{cases}
 \hat {\bf N}_1&=\hat \aa+\hat \aa \hat {\bf M} {\bf N}_2 \\
  {\bf N}_2  &=   \aa {\bf M} \hat {\bf N}_1,
\end{cases}
 \eeq
where $\aa$ and $\hat \aa$ are defined by (\ref{def_A}). This system of linear equations (\ref{matrix_eqns2}) can be solved iteratively with exponential convergence with respect to the matrix norm $||\cdot||_\infty$
\end{theorem}
The proof of this Theorem is very similar to the proof of Theorem \ref{thm_two_sided}, and we leave  the details to the reader.

\section{Ladder processes}\label{ladder}

In this Section we derive several results related to the Laplace exponent $\kappa(q,z)$ of the bivariate ladder process $(L,H)$. We are interested in 
numerical evaluation of this object since it is the key to many important fluctuation identities, such as the quintuple law at the first passage 
which was introduced in \cite{Doney}. 

 Everywhere in this section we will denote the characteristic exponent of the L\'evy process $X$ 
by $\Psi(z)=-\ln \e [\exp(\i z X_1)]$. Note, that the Laplace exponent $\psi(z)$ can be expressed in terms of the characteristic exponent 
as $\psi(z)=-\Psi(-\i z)$, this fact  easily follows from (\ref{eq1}). 

The following Theorem, which holds for general L\'evy processes, is a generalization of the expression for the Wiener-Hopf factors which can be found in Lemma 4.2 in \cite{Mordecki}.
\begin{theorem}\label{thm_Darling} Assume that  for some $\epsilon_1>0$ and for all $|z| \in \r$ large enough we have
 \beq\label{assumptn_on_Psi0}
 |\Psi(z)| > c |z|^{\epsilon_1},
 \eeq
 for some $c>0$,
 and 
\beq\label{assmptn_on_Psi}
 \int\limits_{-\epsilon_2}^{\epsilon_2} \bigg | \frac{\Psi(z)}{z} \bigg | \d z,
\eeq
 exists for some (and then for all) $\epsilon_2>0$. 
 Then for $\re(q)>0$ and $\re(z)>0$ we have
 \beq\label{kappa_int_formula}
 \kappa(q,z)=\exp\left[\frac{1}{2\pi \i} \int\limits_{\r} 
 \left(\frac{\ln(q+\Psi(u))}{u-\i z}-\frac{\ln(1+\Psi(u))}{u} \right) \d u  \right].
 \eeq
\end{theorem}
\begin{proof}
We start with the following integral representation for $\kappa(q,z)$  (see Theorem 6.16 in \cite{Kyprianou})
 \beqq
 \kappa(q,z)= \exp\left[ \;\int\limits_{\r^+} \int\limits_{\r^+} 
 \left(e^{-t}-e^{-q t-z x} \right) \frac{1}{t} \p(X_t \in \d x) \d t \right].
 \eeqq
Assume that the process $X_t$ has a non-zero Gaussian component ($\sigma \ne 0$), then $\p(X_t \in \d x)$ 
 has a density which can be obtained as the inverse
Fourier transform of the right hand side of (\ref{eq1}). Using this fact and assuming that $\gamma>0$ we obtain
 \beq\label{proof_Darling1} \nonumber 
 &&\int\limits_{\r^+} \int\limits_{\r^+} 
 \left(e^{-t-\gamma x}-e^{-q t-z x} \right) \frac{1}{t}\p(X_t \in \d x)  \d t \\ 
  &=&
\frac{1}{2\pi} \int\limits_{\r^+} \int\limits_{\r^+} 
 \left(e^{-t-\gamma x}-e^{-q t-z x} \right) \frac{1}{t}  \int\limits_{\r} e^{-t\Psi(u)-\i u x} \d u \d x \d t.
 \eeq
Applying Fubini Theorem and performing integration in  $x$  we find that the above integral is equal to 
\beq\label{proof_Darling2}
  \lefteqn{
\frac{1}{2\pi} \int\limits_{\r^{\plus}} \frac{\d t}{t} \int\limits_{\r}  \d u
  \left(\frac{e^{-t(1+\Psi(u))}}{\gamma+\i u} -\frac{e^{-t (q+\Psi(u))}}{z+\i u} \right)}&&  \nonumber\\ 
&=& 
 \frac{1}{2\pi} \int\limits_{\r^{\plus}}  \frac{\d t}{t}  \int\limits_{\r}  \d u
 \left(\frac{1}{\gamma+\i u}(e^{-t(1+\Psi(u))}-e^{-t}) -\frac{1}{z+\i u}(e^{-t (q+\Psi(u))}-e^{-t}) \right) ,
\eeq
where in the last step we have used the fact that  for $\gamma >0$ and $\re(z)>0$
 \beqq
 \int\limits_{\r} \left[\frac{1}{\gamma+\i u}-\frac{1}{z+\i u}\right] \d u=0 ,
 \eeqq
(which can be easily verified by residue calculus). Next, we apply Fubini Theorem and Frullani integral to the last integral
 in (\ref{proof_Darling2}) and combining this result with (\ref{proof_Darling1}) we obtain
\begin{equation}
\int\limits_{\r^+} \int\limits_{\r^+} 
 \left(e^{-t-\gamma x}-e^{-q t-z x} \right) \frac{1}{t}\p(X_t \in \d x)  \d t= \frac{1}{2\pi \i} \int\limits_{\r} 
 \left(\frac{\ln(q+\Psi(u))}{u-\i z}-\frac{\ln(1+\Psi(u))}{u-\i \gamma} \right) \d u .
\label{proof_Darling3}
 \end{equation}
Next we need to take the limit of the above identity as $\gamma \to 0^+$. For $u$ small the integrand in the right hand side of 
(\ref{proof_Darling3}) is bounded by
\beqq
C  \frac{\ln(1+|\Psi(u)|}{|u|},
\eeqq
uniformly in $\gamma \in [0,\tilde \gamma]$ for each $\tilde\gamma>0$, and for $u$ large it can be bounded by
\beqq
\bigg| \frac{\ln(q+\Psi(u))}{u-\i z}-\frac{\ln(1+\Psi(u))}{u-\i \gamma} \bigg| &=& \bigg| \i (z-\gamma) \frac{\ln(q+\Psi(u))}{(u-\i \gamma)(u-\i z)}-\frac{\ln\left(1+\frac{1-q}{q+\Psi(u)}\right)}{u-\i \gamma}\bigg|\\&=&O(\ln(|\Psi(u)|) u^{-2})+O((u\Psi(u))^{-1}),
\eeqq
again uniformly in $\gamma \in [0,\tilde \gamma]$. Using the above two estimates and assumption (\ref{assumptn_on_Psi0}) we conclude that the integrand in the right hand side of 
(\ref{proof_Darling3}) is bounded by
\beqq
 \tilde C  \frac{\ln(1+|\Psi(u)|)}{|u|} (1+|u|)^{-\epsilon_1},
\eeqq 
for some $\tilde C>0$ uniformly in $\gamma \in [0,\tilde \gamma]$, and the above function is integrable on $\r$ due to assumption 
(\ref{assmptn_on_Psi}). Thus we can apply the Dominated Convergence Theorem on the right hand side of (\ref{proof_Darling3}) together with the Monotone Convergence Theorem on the left hand side whilst taking limits as $\gamma \to 0^+$. To finish the proof we only 
need to take the limit $\sigma \to 0^+$, which can be justified in exactly the same way using the Dominated Convergence Theorem on the right hand side of (\ref{proof_Darling3}) and weak convergence on the right hand side of (\ref{proof_Darling3}). Note in particular that the characteristic exponent of $X$ is continuous in the Gaussian coefficient and hence, by the Continuity Theorem for Fourier transforms, the law of $X$ is weakly continuous in the Gaussian coefficient.
\end{proof}

Assumption (\ref{assmptn_on_Psi}) is a very mild one. It is satisfied by all processes except those which have unusually heavy tails 
of the L\'evy measure. It is satisfied by all processes in $M$-class, and more generally, by all processes for which there exists an $\epsilon>0$ such that $\Pi( \r \setminus (-x,x) )=O(x^{-\epsilon})$ as $x\to +\infty$ (for example, all stable processes have this property). While condition (\ref{assumptn_on_Psi0}) is more restrictive, one can see that it only excludes compound Poisson processes and some processes of bounded variation 
which are equal to the sum of its jumps. One example of such a process is a pure jump Variance Gamma process with no linear drift 
 $X_t=\sigma W_{\Gamma_t}+\mu \Gamma_t$, which satisfies  $|\Psi(z)| \sim c \ln(|z|)$ as $z\ \to \infty$. 
 One can see that condition (\ref{assumptn_on_Psi0}) is satisfied for the processes of bounded variation with non-zero drift, processes 
with non-zero Gaussian component and all processes with the L\'evy measure satisfying $\Pi(\r \setminus (-x,x)) > x^{-\epsilon}$ as $x\to 0^+$ for some $\epsilon >0$. In particular this condition is satisfied for all processes in the  $\beta$-, $\theta$-,  hypergeometric or hyperexponential family excluding compound Poisson processes.

Theorem \ref{thm_Darling} allows us to derive an expression for the  Laplace exponent of the bivariate ladder process $(L,H)$. 

\begin{theorem}\label{thm_kappa} Assume that $X$ belongs to $\beta$-, $\theta$-, hyper-exponential or hypergeometric family of Levy processes and that $X$ is not 
a compound Poisson process. Then
 \beq\label{eq_kappa_phiqp}
 \kappa(q,z)= \left[ \phiqp(\i z)\right]^{-1} \prod\limits_{n\ge 1} \frac{\zeta_n(q)}{\zeta_n(1)}.
 \eeq
\end{theorem}
\begin{proof}
 If $X$ belongs to $\beta$-, $\theta$- or hypergeometric family of Levy processes, we can use asymptotics for $\zeta_n$ as $n\to +\infty$ 
 (see \cite{Kuz-beta}, \cite{Kuz-theta} and \cite{KyPavS}) to find that $ \zeta_n(q)/\zeta_n(1)=1+O(n^{-1-\epsilon})$ and 
$\hat \zeta_n(q)/\hat\zeta_n(1)=1+O(n^{-1-\epsilon})$,  which implies that  both infinite products
 \beqq
\prod\limits_{n\ge 1} \frac{\zeta_n(q)}{\zeta_n(1)}, \qquad \prod\limits_{n\ge 1} \frac{\hat \zeta_n(q)}{\hat \zeta_n(1)},
 \eeqq
converge. Next, using the same technique as in the proof of Lemma 6  in \cite{Kuz-beta}, one can prove that
\beqq
  \prod\limits_{n\ge 1} \frac{1-\frac{\i z}{\zeta_n(q)}}{1-\frac{\i z}{\zeta_n(1)}}
\prod\limits_{n\ge 1} \frac{1+ \frac{\i z}{\hat \zeta_n(q)}}{1+ \frac{\i z}{\hat \zeta_n(1)}} \to 
 \prod\limits_{n\ge 1} \frac{\zeta_n(1)}{\zeta_n(q)} \prod\limits_{n\ge 1} \frac{\hat \zeta_n(1)}{\hat \zeta_n(q)},
\eeqq
as $z\to \infty$, $z\in \r$. Using the Wiener-Hopf factorization $q/(q+\Psi(z))=\phiqp(z) \phiqm(z)$ and (\ref{eqn_WH_factor}) we obtain
\beqq
\frac{q+\Psi(z)}{1+\Psi(z)}=
q 
\prod\limits_{n\ge 1} \frac{1-\frac{\i z}{\zeta_n(q)}}{1-\frac{\i z}{\zeta_n(1)}}
\prod\limits_{n\ge 1} \frac{1+ \frac{\i z}{\hat \zeta_n(q)}}{1+ \frac{\i z}{\hat \zeta_n(1)}}.
\eeqq
Since $X$ is not a compound Poisson process, we have $|\Psi(z)| \to \infty$ as $z \to \infty$, $z\in \r$, thus we finally  conclude that
\beqq
\prod\limits_{n\ge 1} \frac{\zeta_n(q)}{\zeta_n(1)} \prod\limits_{n\ge 1} \frac{\hat \zeta_n(q)}{\hat \zeta_n(1)}=q. 
\eeqq
The next step is to use the Wiener-Hopf factorization $q/(q+\Psi(z))=\phiqp(z) \phiqm(z)$, (\ref{eqn_WH_factor}) and the above identity to 
 rewrite $q+\Psi(z)$ as 
 \beqq
 q+\Psi(z)=
\prod\limits_{n\ge 1} \frac{ \frac{\zeta_n(q)-\i z}{\zeta_n(1)}}{1-\frac{\i z}{\rho_n}}
\prod\limits_{n\ge 1} \frac{\frac{\hat \zeta_n(q) + \i z}{\hat \zeta_n(1)}}{1+\frac{\i z}{\hat \rho_n}}.
 \eeqq
Now we can use the above factorization and the following integral identity (which can be proved by shifting the contour of integration in the complex plane) 
 \beqq
 \frac{1}{2\pi \i} \int\limits_{\r} \left(\frac{\ln(\frac{1}{b}(a-u))}{u- \i z}-\frac{\ln(\frac{1}{b}(b-u))}{u} \right) \d u=
 \begin{cases}
\ln\left(\frac{1}{b}(a-\i z)\right), \; &\textnormal{ if  } \im(a)<0, \; \im(b)<0 \\
0, \; &\textnormal{ if  } \im(a)>0, \; \im(b)>0,
 \end{cases}
 \eeqq 
to deduce that
 for $\re(z)>0$ 
 \beqq
\frac{1}{2\pi \i} \int\limits_{\r} 
 \left(\frac{\ln(q+\Psi(u))}{u-\i z}-\frac{\ln(1+\Psi(u))}{u} \right) \d u = \sum\limits_{n\ge 1} 
 \left[ \ln \left(  \frac{\zeta_n(q)+z}{\zeta_n(1)}\right) - \ln\left( 1+\frac{z}{\rho_n} \right) \right],
 \eeqq
which is equivalent to (\ref{eq_kappa_phiqp}).
\end{proof}

\begin{corollary}\label{corollary_bivariate_Pi}
  Let $\nu(\d s, \d x)$ be the L\'evy measure of the ascending ladder process $(L_t^{-1},H_t)$. Then for $x>0$ we have
 \beq\label{eqn_ladder_pr_Levy_measure}
\int\limits_{\r^{\plus}} e^{-q s} \nu (\d s, \d x)=\left[ \; \prod\limits_{n\ge 1} \frac{\zeta_n(q)}{\zeta_n(1)} \right]
   \bar\coeffb(\zeta,\rho)^T \times \bar\vv(\rho,x)  \d x.
 \eeq 
\end{corollary}
\begin{proof}
 Formula (\ref{eqn_ladder_pr_Levy_measure}) is a corollary of (\ref{eq_kappa_phiqp}), (\ref{see-the-drift}) and the  fact that $\nu(\d s, \d x)$ is related to $\kappa(q,z)$ 
 through the formula 
\beqq
 \kappa(q,z)=\kappa(q,0)+az +\int\limits_{0}^{ \infty} \left(1 - e^{ - z x} \right) \int\limits_{0}^{ \infty}e^{-qs}\nu(\d s, \d x).
\eeqq
\end{proof}

\section{More Fluctuation identities}\label{more}
We offer some more fluctuation identities. Although they are slightly more complex, they are still equally straightforward for the purpose of numerical work.

We assume throughout this section that $X$ is regular for both $(0,\infty)$ and $(-\infty,0)$. Equivalently, we assume that $\coeffa_0(\rho,\zeta) = \coeffa_0(\hat\rho, \hat\zeta)=0$. This is the case if, for example, $X$ has paths of unbounded variation. It will be clear from the proofs of the results given below how this assumption may be removed.

For $a>0$ and for $y \le a$ we define   resolvent for $X$ killed on leaving $(-\infty, a]$ as
\[
R^{(q)}(a, {\rm d}y) : = \int_0^\infty e^{-qt} \mathbb{P} (X_t\in {\rm d}y; t< \tau^+_a){\rm d}t.
\]

\begin{theorem}\label{thm_Rq}
Define a matrix $\dd=\{d_{i,j}\}_{i,j\ge 0}$ as follows
 \beqq
 d_{i,j}=
 \begin{cases}
  0\;\;\; & {\textnormal {if  }}  i = 0 {\textnormal { or }}  j=0 \\
 \displaystyle \coeffa_i(\rho,\zeta) \frac{ 1}{\zeta_i+ \hat \zeta_j}
 \coeffa_j(\hat \rho, \hat \zeta)    \;\;\; &{\textnormal {if  }} i\ge 1, \; j \ge 1 .
 \end{cases}
 \eeqq
Then if $y \le a$  we have
 \beqq
 qR^{(q)}(a, {\rm d}y)&=&  \left[ \bar \vv (\zeta,0 \vee y)^T \times \dd \times \bar \vv(\hat \zeta, 0 \vee (-y)))- 
 \bar \vv(\zeta,a)^T \times \dd \times \bar \vv(\hat \zeta,a-y) \right] \d y.
 \eeqq
\end{theorem}

\begin{proof}

From the proof of  Theorem 20  on p176 of \cite{Bertoin}, it can be seen that 
\beqq
q\int\limits_{(-\infty,a]} f(y)R^{(q)}(a, {\rm d}y) 
&=& 
 \int\limits_{[0,a]} \mathbb{P}(\overline{X}_{\ee(q)} \in{\rm d}z) \int\limits_{[0,\infty)}  \mathbb{P}(-\underline{X}_{\ee(q)} \in{\rm d}u)f(-u+z)\\
 &=&  \int\limits_{[0,a]} \mathbb{P}(\overline{X}_{\ee(q)}  \in{\rm d}z) 
                     \int\limits_{-(\infty,z]} f(y) \p(\underline{X}_{\ee(q)}  \in  -z+\d y)\\
&=&  \int\limits_{(-\infty,a]}  f(y)\int\limits_{[0\vee y,a]} \mathbb{P}(\overline{X}_{\ee(q)} \in \d z)  
\p(\underline{X}_{\ee(q)}  \in -z+\d y).
\eeqq
Thus we obtain an alternative representation of the Spitzer-Bertoin identity 
\beqq
q R^{(q)}(a, {\rm d}y)= \int\limits_{[0\vee y,a]} \mathbb{P}(\overline{X}_{\ee(q)} \in \d z)  
\p(\underline{X}_{\ee(q)}  \in -z+\d y).
\eeqq
To finish the proof we have to use the formulas in (\ref{dist_sup_X}) and perform the integration in the above expression (noting that some terms are lost on account of the fact that we have assumed $\coeffa_0(\rho,\zeta) = \coeffa_0(\hat\rho, \hat\zeta)=0$).
\end{proof}

Next define 
\beqq
\Theta^{(q)}(a, x, {\rm d}y)  = \int_0^\infty e^{-q t}\mathbb{P}_x(X_t \in {\rm d}y, t<\tau^+_a \wedge \tau^-_0) \d t.
\eeqq
\begin{theorem}
Assume $q>0$ and $y\in [0,a]$, then
\beqq
\lefteqn{q\Theta^{(q)}(a, x, {\rm d}y)} &&\\
&&=\left[ \bar \vv (\zeta,0 \vee y-x)^T \times \dd \times \bar \vv(\hat \zeta, 0 \vee (x-y))- 
 \bar \vv(\zeta,a-x)^T \times \dd \times \bar \vv(\hat \zeta,a-y) \right] \d y \\
 &&- \left[  \bar \vv(\hat \zeta,x)^T \times \hat \cc_1 + \bar \vv(\zeta,a-x)^T \times \hat \cc_2 \right] 
\big [\bb(y) \times \dd \times \bar \vv(\hat\zeta,0) - \bb(a) \times \dd \times \bar \vv(\hat\zeta,a-y) 
 \big ] {\rm d}y.
\eeqq
where matrix $\bb(y)=\bb(\hat \rho,\zeta,y)$  is defined in (\ref{def_B}), $\dd$ is defined in Theorem \ref{thm_Rq} while matrices $\hat \cc_1$ and $\hat \cc_2$ come from Theorem  \ref{thm_two_sided}. 
Also, $\bar \vv(\hat\zeta,0)$ is intepreted as $[0,\hat\zeta_1,\hat\zeta_2,...]^T$. 
\end{theorem}

\begin{proof}
Define 
\beq\label{def_gq}
\hat g^{(q)}(a,x, {\rm d}y) = \mathbb{E}_x\left[ e^{-q \tau^-_0} \ind(-X_{\tau^-_0} \in {\rm d}y  \; ; \;  \tau^-_0 <\tau^+_a ) \right].
\eeq
Note moreover that for $f$ supported in $[0,a]$,
\beqq
\int_{[0,a]} f(y)\Theta^{(q)}(a, x, {\rm d}y) &=&
\frac{1}{q}\e_x\left[ f(X_{\ee(q)}) \ind(\ee(q)<\tau_0^- \wedge \tau_a^+) \right]\\
&=& \frac{1}{q}\e_x\left[ f(X_{\ee(q)}) \ind(\ee(q)< \tau_a^+) \right]-\frac{1}{q}\e_x\left[ f(X_{\ee(q)}) \ind(\tau_0^-< \ee(q)< \tau_a^+) \right]\\
&=& \int_0^a f(y) R^{(q)}(a-x,\d y-x)-\frac{1}{q}\e_x\left[  f(X_{\ee(q)}) \ind(\tau_0^-< \ee(q)< \tau_a^+) \right].
\eeqq
The second expectation in the above expression can be rewritten as
\beqq
 \lefteqn{
 \int_0^\infty e^{-q t}\mathbb{E}_x\left[ f(X_t) \ind(\tau_0^-< t < \tau_a^+) \right]{\rm d}t}&&\\
&=&\mathbb{E}_x\left[
\int_{\tau^-_0}^\infty e^{-qt} f(X_t)\ind(t<\tau^+_a ) {\rm d}t \right]\\
&=&\mathbb{E}_x\left[e^{-q\tau^-_0}\ind(\tau^-_0<\tau^+_a )
\mathbb{E}_{X_{\tau^-_0}}\left[\int_{0}^\infty e^{-qt} f(X_t)\ind(t<\tau^+_a ){\rm d}t
\right]\right]\\
&=& \int\limits_{\r^+} \hat g^{(q)}(a,x,\d z) \int\limits_{\r^+} R^{(q)}(a+z,z+\d y) f(y).
\eeqq
So in conclusion we have for $y\in[0,a]$,
\beqq
\Theta^{(q)}(a, x, {\rm d}y) =R^{(q)}(a-x, {\rm d}y-x)- \int\limits_{\r^+}  \hat g^{(q)}(a,x,\d z) R^{(q)}(a+z,z+\d y) ,
\eeqq
and to end the proof one should use results of Theorems \ref{thm_two_sided} and \ref{thm_Rq} and compute the above integral. The details are left to the reader.
\end{proof}

\begin{remark} The previous result also allows us to write down the discounted joint overshoot, undershoot distribution for the two-sided exit problem:
\beqq
g^{(q)}(a, x, {\rm d}y, {\rm d}z) =\e_x\left[ e^{- q\tau^+_a} \ind( \tau^+_a <\tau^-_0 \; ; \;  X_{\tau^+_a}\in {\rm d}y \;;\; X_{\tau^+_a -}\in{\rm d}z )\right].
\eeqq
 (Again this relates to the so-called Gerber-Shiu measure for classical risk theory). 
Indeed, for $y>a$ and $z\in[0,a]$, using the compensation formula we have
\beqq
g^{(q)}(a, x, {\rm d}y, {\rm d}z)&=&\e_x\left[ \sum_{t\geq 0}e^{-qt} \ind( \overline{X}_{t-}\leq a,\, \underline{X}_{t-}\geq 0,\,  X_{t-}\in{\rm d}z )
 \ind( X_t\in{\rm d}y )\right]\\
&=&\e_x\left[\int\limits_{\r^+} e^{-qt} \ind ( \overline{X}_{t-}\leq a,\, \underline{X}_{t-}\geq 0,\,  X_{t-}\in{\rm d}z ) {\rm d}t \right]\Pi({\rm d}y -z)\\
&=&\Theta^{(q)}(a, x, {\rm d}z)\Pi({\rm d}y -z).
\eeqq
\end{remark}

\begin{remark}
Define the reflected process $Y: = X-\underline{X}$ and its resolvent when killed on exiting $[0,a]$,
\beqq
\Lambda^{(q)}(a, x,{\rm d}y) = \int_0^\infty e^{-q t}\mathbb{P}_x(Y_t \in {\rm d}y, \,t<\sigma_a ) \d t,
\eeqq 
where $\sigma_a = \inf\{s>0: Y_s >a \}$. According to \cite{Baurdoux2009}  one may write for all $q\geq 0$,
\[
\Lambda^{(q)}(a, x, {\rm d}y) = \Theta^{(q)}(a, x,{\rm d}y) + \hat g^{(q)}(a, x)\cdot\lim_{z\downarrow 0}
\frac{\Theta^{(q)}(a,z, {\rm d}y)}{1-\hat g^{(q)}(a, z)}.
\]
where $\hat g^{(q)}(a, x) = \mathbb{E}_x [ e^{- q\tau^-_0} \ind( \tau^+_a >\tau^-_0) ]$.
Moreover, note that we may continue the computations in a similar way to before with the help of the compensation formula,
\[
\mathbb{E}_0 \left[ e^{-q\sigma_a} \ind( Y_{\sigma_a}\in{\rm d}y \; ; \;  Y_{\sigma_a -}\in{\rm dz}) \right]=
\Lambda^{(q)}(a, 0, {\rm d}z)\Pi({\rm d}y -z),
\]
and hence
\[
\mathbb{E}_x \left[ e^{-q\sigma_a}\ind( Y_{\sigma_a}\in{\rm d}y \; ; \;  Y_{\sigma_a -}\in{\rm dz})\right] = 
g^{(q)}(a, x, {\rm d}y, {\rm d}z) + \hat g^{(q)}(a, x)\Lambda^{(q)}(a, 0, {\rm d}z)\Pi({\rm d}y -z).
\]
\end{remark}

\section{Numerical Results}\label{sec_numerics}
 For all our numerical examples we will use a process $X$ from the $\beta$-family (see the introduction or \cite{Kuz-beta}) having parameters
 \beqq
  (\sigma, \mu, \alpha_1,\beta_1,\lambda_1,c_1,\alpha_2,\beta_2,\lambda_2, c_2) = (\sigma, \mu, 1, 1.5, 1.5, 1, 1, 1.5, 1.5 , 1).
 \eeqq
Here $\mu=\e[ X_1]$ and $\sigma$ is the Gaussian coefficient, the other parameters define the density of a L\'evy measure, which has exponentially decaying tails and $O(|x|^{-3/2})$ singularity at $x=0$, thus this process has jumps of infinite activity but finite varation. We define the
 following four parameter sets
 \beqq
&& {\textnormal{ Set 1: }} \sigma=0.5, \mu=1  \qquad  {\textnormal{ Set 2: }} \sigma=0.5,  \mu=-1 \\
&& {\textnormal{ Set 3: }} \sigma=0,  \mu=1 \qquad \;\;\; {\textnormal{ Set 4: }} \sigma=0, \mu=-1.
 \eeqq

 As the first illustration of the efficiency of our algorithms we will compute the following three quantities related to the double exit problem: 
\begin{itemize}
 \item[(i)]  density of the overshoot one the event that the process exists at the upper boundary
 \beqq
 f_1(x,y)=\frac{\d}{\d y} \; \e_x \left[ e^{-q \tau_1^+} \ind \left(X_{\tau_1^+} \le  y \; ; \; \tau_1^+<\tau_0^- \right)\right],
 \eeqq 
 \item[(ii)]  probability of exiting from the interval $[0,1]$ at the upper boundary
  \beqq
 f_2(x)=\e_x \left[ e^{-q \tau_1^+} \ind \left( \tau_1^+<\tau_0^- \right)\right],
 \eeqq
 \item[(iii)] probability of exiting the interval $[0,1]$ by creeping across the upper boundary
  \beqq
 f_3(x)=\e_x \left[ e^{-q \tau_1^+} \ind \left(X_{\tau_1^+}=1 \; ; \; \tau_1^+<\tau_0^- \right)\right].
 \eeqq
\end{itemize}
In order to compute these expressions we use methods described in Theorem \ref{thm_two_sided}. We truncate all the matrices $\cc_i, \hat \cc_i$ so that they have 
 size $200 \times 100$ (this corresponds to truncating coefficients $\coeffa_i(\rho,\zeta)$ and $\coeffa_i(\hat\rho,\hat\zeta)$ 
at $i=200$ and coefficients $\coeffb_j(\zeta,\rho)$ and $\coeffb_j(\hat \zeta,\hat \rho)$ at $j=100$). In order to compute coefficients 
$\coeffa_i(\rho,\zeta)$, $\coeffa_i(\hat\rho,\hat\zeta)$, $\coeffb_j(\zeta,\rho)$ and $\coeffb_j(\hat \zeta,\hat \rho)$ we truncate infinite products in 
(\ref{formula_cn_a_b}) and (\ref{formula_cn_b_a}) at $k=400$, thus all the computations depend on precomputing $\{\zeta_n,\hat \zeta_n\}$ for $n=1,2,..,400$.
All the code was written in Fortran and  the computations were performed on a
standard laptop (Intel Core 2 Duo 2.5 GHz processor and 3 GB of RAM).

We present the results for $q=1$ in Figures \ref{fig_Set12} and  \ref{fig_Set34}. Computations required to produce graphs for each parameter set took around 0.15 seconds. 
The numerical results clearly show the effects that we would expect to see. In Figures  \ref{fig_set1_1d}, \ref{fig_set2_1d} and \ref{fig_set3_1d}
 we see a positive probability of creeping, which is expected since the process $X$ has a Gaussian component in the first two cases and a bounded variation
and positive drift in the third case. Parameter Set 4 corresponds to a process with bounded variation and negative drift, thus we do not have any upward creeping, 
and this is exactly what we obtain in Figure \ref{fig_set4_1d}. Also, figures \ref{fig_set1_2d}, \ref{fig_set2_2d} and \ref{fig_set3_2d} show that
$f_1(x,y) \to 0$ as $x\to 1^-$, which confirms our expectation, as in this case the upper half line is regular and as $x\to 1^-$ the process will cross the barrier at 1 
by creeping, not by jumping over it. This is different from figure   \ref{fig_set4_2d}, where, because of bounded variation and negative drift, the upper half line is 
irregular and the only way to cross the barrier at 1 is by jumping over it. Next, Figures \ref{fig_set1_1d} and \ref{fig_set2_1d} show that $f_i(x) \to 0$ as $x\to 0^+$ and $f_i(x)\to 1$ as $x\to 1^-$, which again agrees with the theory, as for parameter Sets 1 and 2 the process $X$ has a Gaussian
component, therefore $0$ is regular for $(-\infty,0)$ and $(0,\infty)$. This is not so for parameter Sets 3 and 4, since now the process $X$ has bounded variation
and drift, and depending on the sign of the drift, $0$ is regular for either $(0,\infty)$ or $(-\infty,0)$, and this is what we observe on figures  
\ref{fig_set3_1d} and \ref{fig_set4_1d}.

For the next example, we compute the density $u(s,x)$ of the bivariate renewal measure ${\mathcal U}(\d s, \d x)$ defined by
 \beqq
 {\mathcal U}(\d s, \d x)=\int\limits_{\r^+} \p(L_t^{-1} \in \d s, H_t \in \d x) \d t ,
 \eeqq
 where $(L,H)$ is the ascending ladder process, see \cite{Doney}, \cite{Kyprianou} and \cite{KyPaRi}. This measure is a very important object, 
 as it gives us full knowledge of the quintuple law at the first passage, see \cite{Doney}. 
 Using formulas 6.18 and 7.10 in \cite{Kyprianou} we see that ${\mathcal U}(\d s, \d x)$ satisfies
\beqq
\int\limits_{\r^+} e^{-q s} {\mathcal U}(\d s, \d x)=\frac{\p(\xsup_{\ee(q)} \in \d x)}{\kappa(q,0)},
\eeqq
therefore using Theorem \ref{thm_kappa} and (\ref{dist_sup_X}) we find that the density of ${\mathcal U}(\d s, \d x)$  can be computed as
\beqq
u(s,x)= \frac{1}{2\pi \i} \int\limits_{q_0+\i \r} \prod\limits_{n\ge 1} \frac{\zeta_n(1)}{\zeta_n(q)} \left[ \bar\coeffa(\rho,\zeta(q))^T \times \bar\vv(\zeta(q),x) \right] e^{qs} \d q,
\eeqq
for any $q_0>0$. It turns out that it is quite easy to compute the above integral numerically using technique discussed in \cite{Kuz-beta} coupled with a Filon-type method (see \cite{Iserles}). Producing each graph on Figure \ref{fig_Vdxds} takes around $1.2$ seconds. In order to compute $u(s,x)$ we truncated  
$\coeffa_i(\rho,\zeta)$ at $i=100$ and used $200$ roots $\zeta_k$ to compute these coefficients using formulas (\ref{formula_cn_a_b}). We chose $q_0=0.25$, truncated integal in $q$ at $|\im(q)|<10^4$. 
The numerical results are presented in Figure \ref{fig_Vdxds}. Note that both the ascending ladder height and time processes behind Figure \ref{fig_Vdxds} (c) have a linear drift where as in Figure \ref{fig_Vdxds} (d) they are both driftless compound Poisson processes. In the former case this explains the strong concentration of mass around a linear trend, and in the latter case there exists an atom at $x=s=0$, which is not visible on the graph since  we are only plotting the absolutely continuous part.

As we have mentioned in the introduction, all expressions related to fluctuation identities presented in this paper have the following property: (i) they are 
computed explicitly in terms of roots/poles of $q+\Psi(\i z)$ and possibly some linear algebra operations, (ii) all of them have the law of the space variables 
(for example, overshoot or location of the last maximum) in closed form and  (iii) they involve Laplace transform of the first passage time $\tau_a^+$ or $\tau_0^-$. The third condition implies that if we want to compute joint distribution of both space and time functionals of the process (for example, joint density of the first passage time
and the overshoot) we would have to perform a Fourier transform in the $q$-variable, and it has to be done numerically. See \cite{Kuz-beta} and \cite{KyPavS} for examples
of application of this technique. It turns out, that using exactly the same method we can also be used to obtain similar reults for L\'evy processes 
with stochastic volatility, which are very popular models in Mathematical Finance, see for example \cite{Carr2003}. 
 We will briefly present this technique for the case of the two-sided exit problem considered above. 
 Let $T_s$ be an increasing continuous process satisfying $T_0=0$. We require that $\e [ \exp(q T_s)]$ is known in closed form, a 
classical example is the integral of the Cox-Ingersoll-Ross diffusion process, however one can also choose several other processes, see \cite{HurdKuz}. Define a time-changed process $Z_s=X_{T_s}$, $s\geq 0$, where we assume that $T_s$ is independent of $X_t$. As before, define
 $s_0^-$   $\{s_a^+\}$ to be the first passage time of process $Z_s$ below $0$ \{above $a$\}. Since $T_s$ is continuous we have $T_{s_a^+}=\tau_a^+$, thus we obtain  for any positive $q_0$
\beqq
\nonumber
&&\e_z \left[  \ind \left(Z_{s_a^+} \in  \d y  \; ; \; s_a^+ \le u \; ; \; s_a^+<s_0^- \right)\right]=
\e_z \left[ \ind \left(X_{\tau_a^+} \in \d y \; ; \; \tau_a^+ \le  T_u \; ; \; \tau_a^+ < \tau_0^-\right) \right]\\ 
&=&
\int\limits_{\r^+} \e_z \left[ \ind \left(X_{\tau_a^+} \in \d y \; ; \; \tau_a^+ \le  t \; ; \; \tau_a^+ < \tau_0^- \right) \right] \p(T_u \in \d t)=
\frac{1}{2\pi \i} \int\limits_{q_0+\i \r} g^{(q)}(a, z, {\rm d} y ) \e \left[ e^{q T_u} \right] q^{-2} \d q,
\eeqq
where we have used the fact that 
\beqq
\e_z \left[ \ind \left(X_{\tau_a^+} \in \d y \; ; \; \tau_a^+ \le  t \; ; \; \tau_a^+ < \tau_0^- \right) \right]=
\frac{1}{2\pi \i } \int\limits_{q_0+\i \r} g^{(q)}(a, z, {\rm d} y ) e^{q t} q^{-2} \d q,
\eeqq
which follows from  (\ref{def_gq}) by inverting the Laplace transform in $q$. Thus we see that if we choose the time change process $\{T_s: t\geq 0\}$ for which the
 Laplace transform $\e [ \exp(q T_s)]$ is known in closed form, then computing quantities for the time-changed process $Z$ is essentially identical to computing the same quantities 
for the process $X$ itself.  

\bibliographystyle{siam}
\bibliography{meromorphic_processes}


\newpage

\begin{figure}
\centering
\subfloat[][Set 1: $f_1(x,y)$]{\label{fig_set1_2d}\includegraphics[height =6cm]{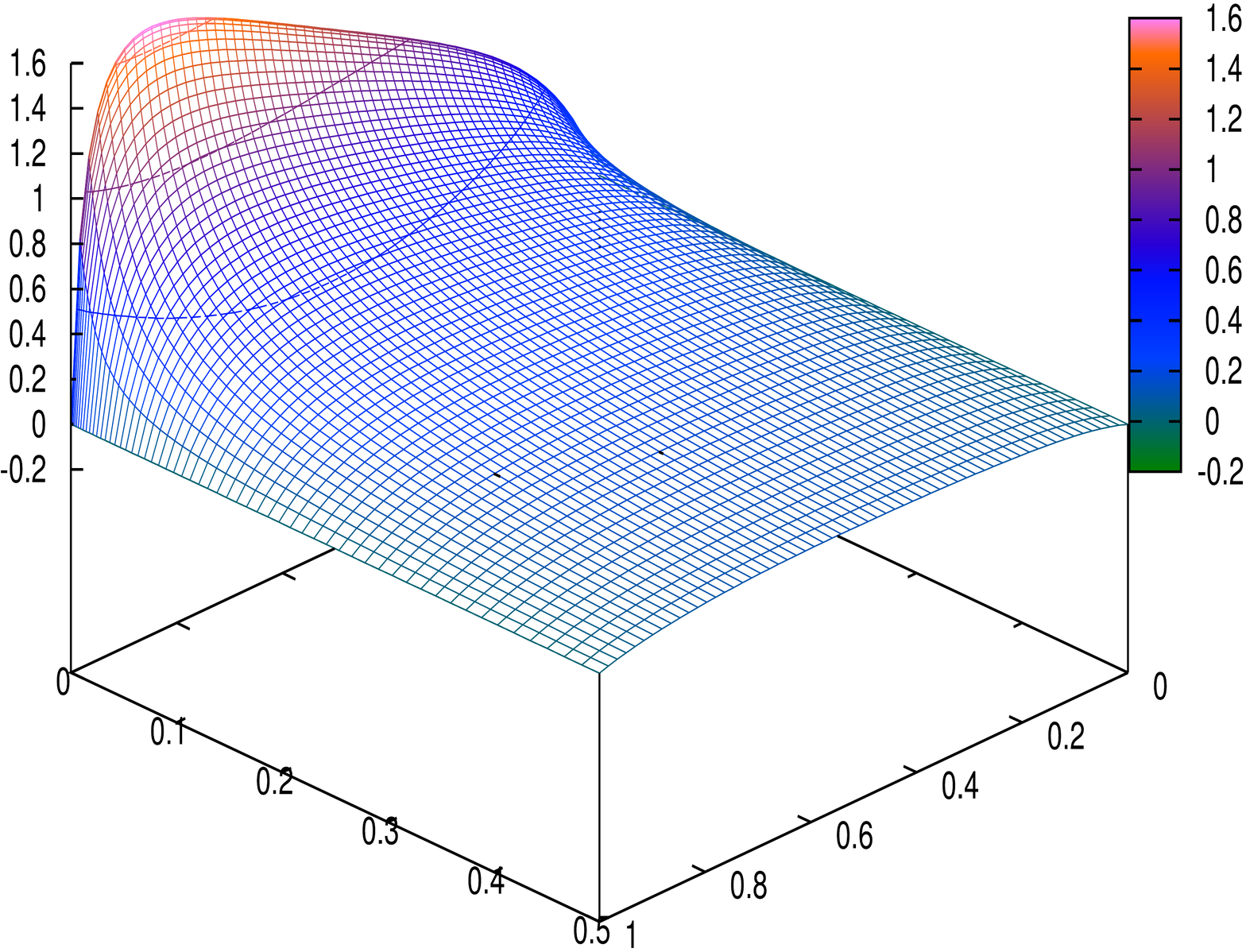}} 
\subfloat[][Set 1:$\; f_2(x)$ (solid) and $f_3(x)$ (circles)]{\label{fig_set1_1d}\includegraphics[height =6cm]{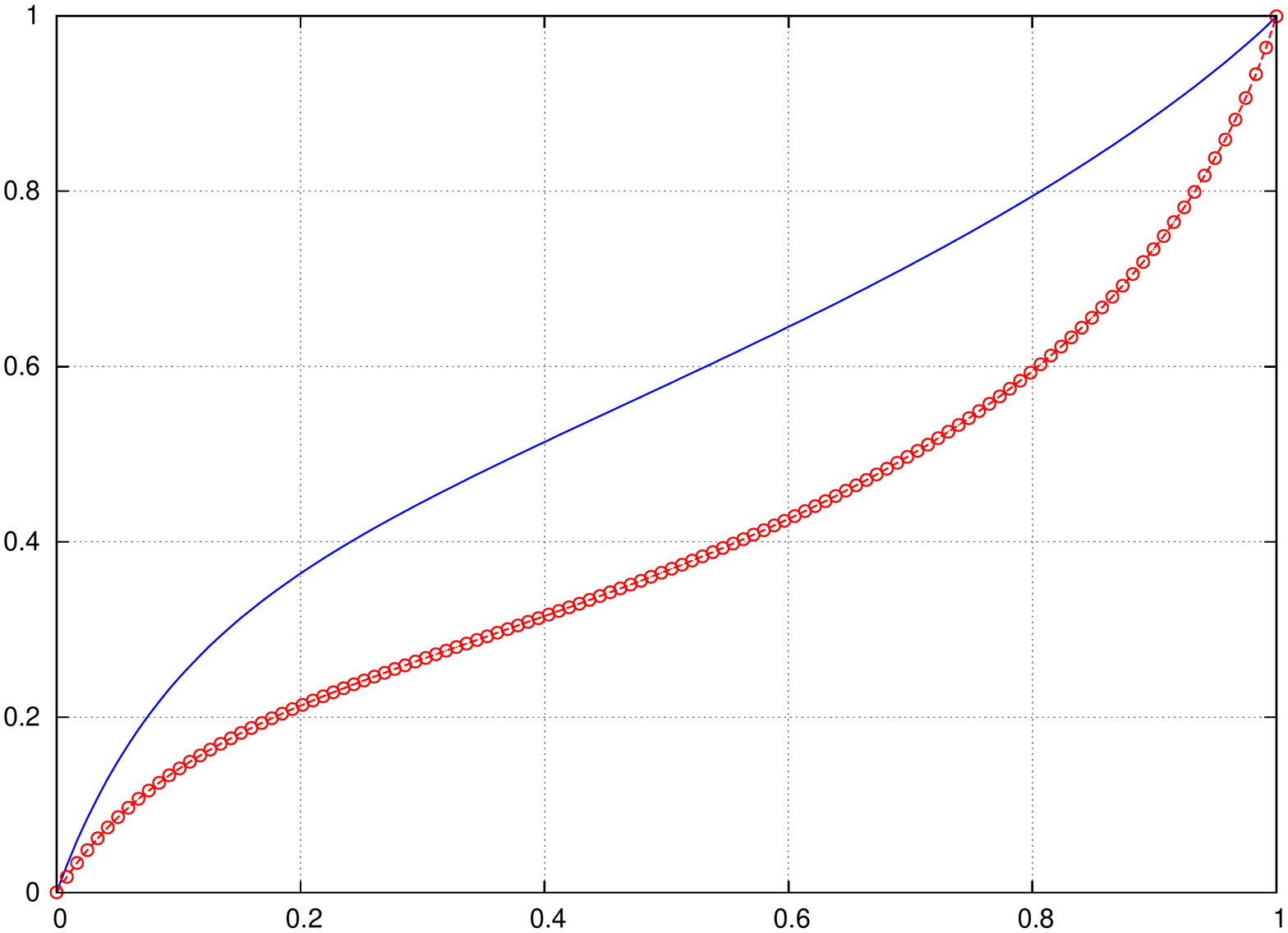}} \\
\subfloat[][Set 2: $f_1(x,y)$]{\label{fig_set2_2d}\includegraphics[height =6cm]{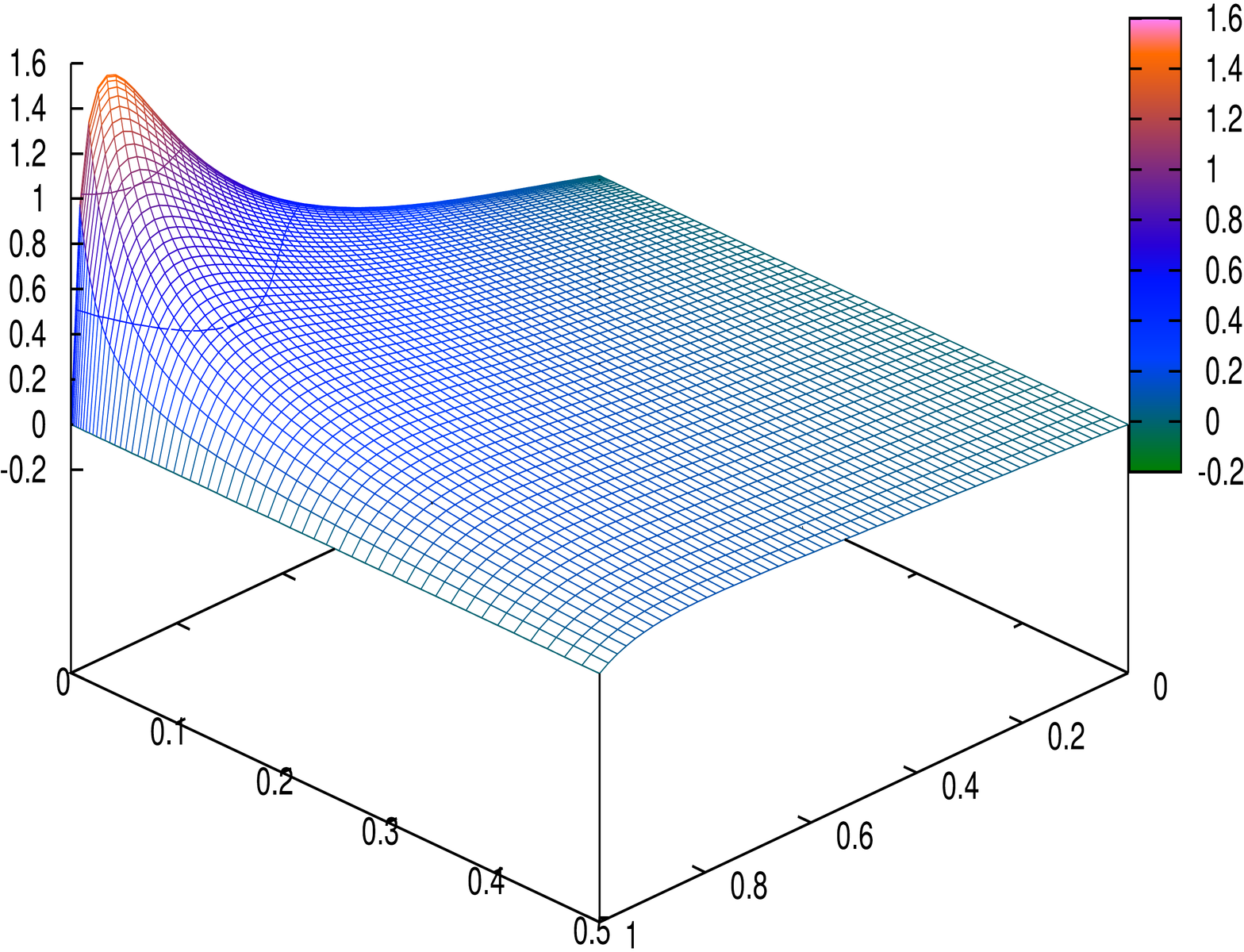}} 
\subfloat[][Set 2:$\; f_2(x)$ (solid) and $f_3(x)$ (circles)]{\label{fig_set2_1d}\includegraphics[height =6cm]{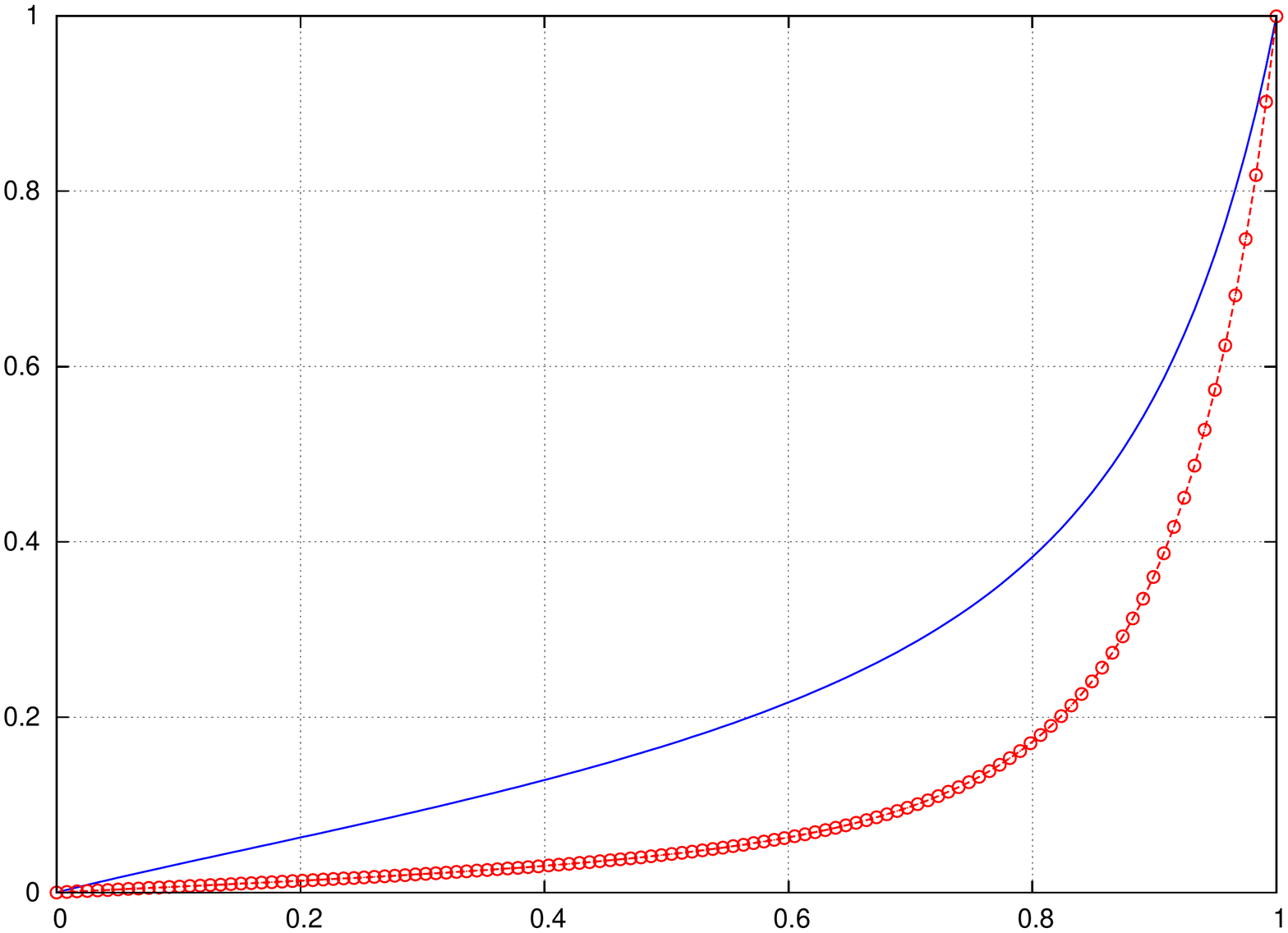}} 
\caption{Unbounded variation case ($\sigma=0.5$): computing the density of the overshoot $f_1(x,y)$ ($x\in (0,1)$, $y\in (0,0.5)$), probability of first exit $f_2(x)$
and probability of creeping $f_3(x)$ for parameter Set 1 (positive drift $\mu=1$) and Set 2 (negative drift $\mu=-1$).} 
\label{fig_Set12}
\end{figure}

\begin{figure}
\centering
\subfloat[][Set 3: $f_1(x,y)$]{\label{fig_set3_2d}\includegraphics[height =6cm]{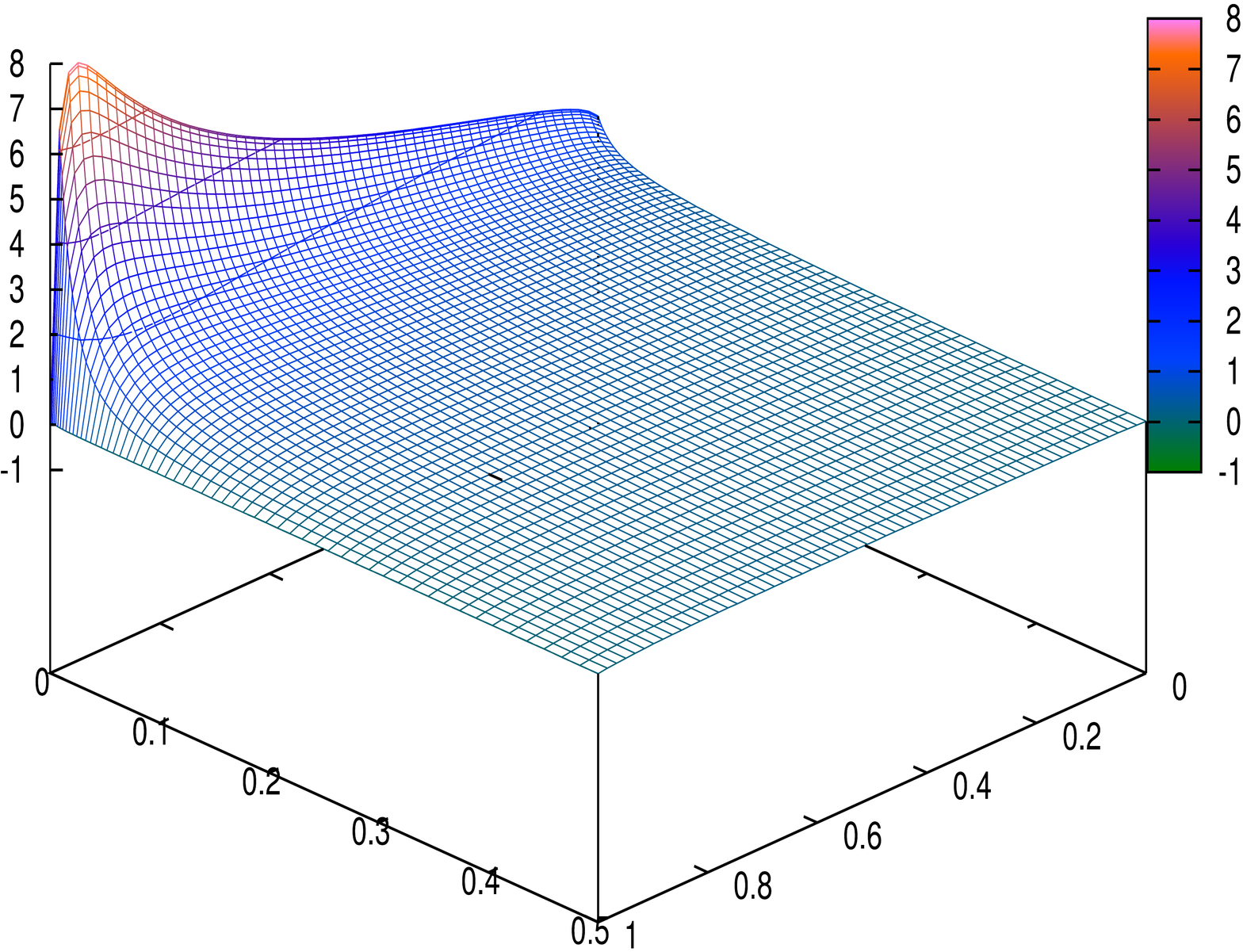}} 
\subfloat[][Set 3:$\; f_2(x)$ (solid) and $f_3(x)$ (circles)]{\label{fig_set3_1d}\includegraphics[height =6cm]{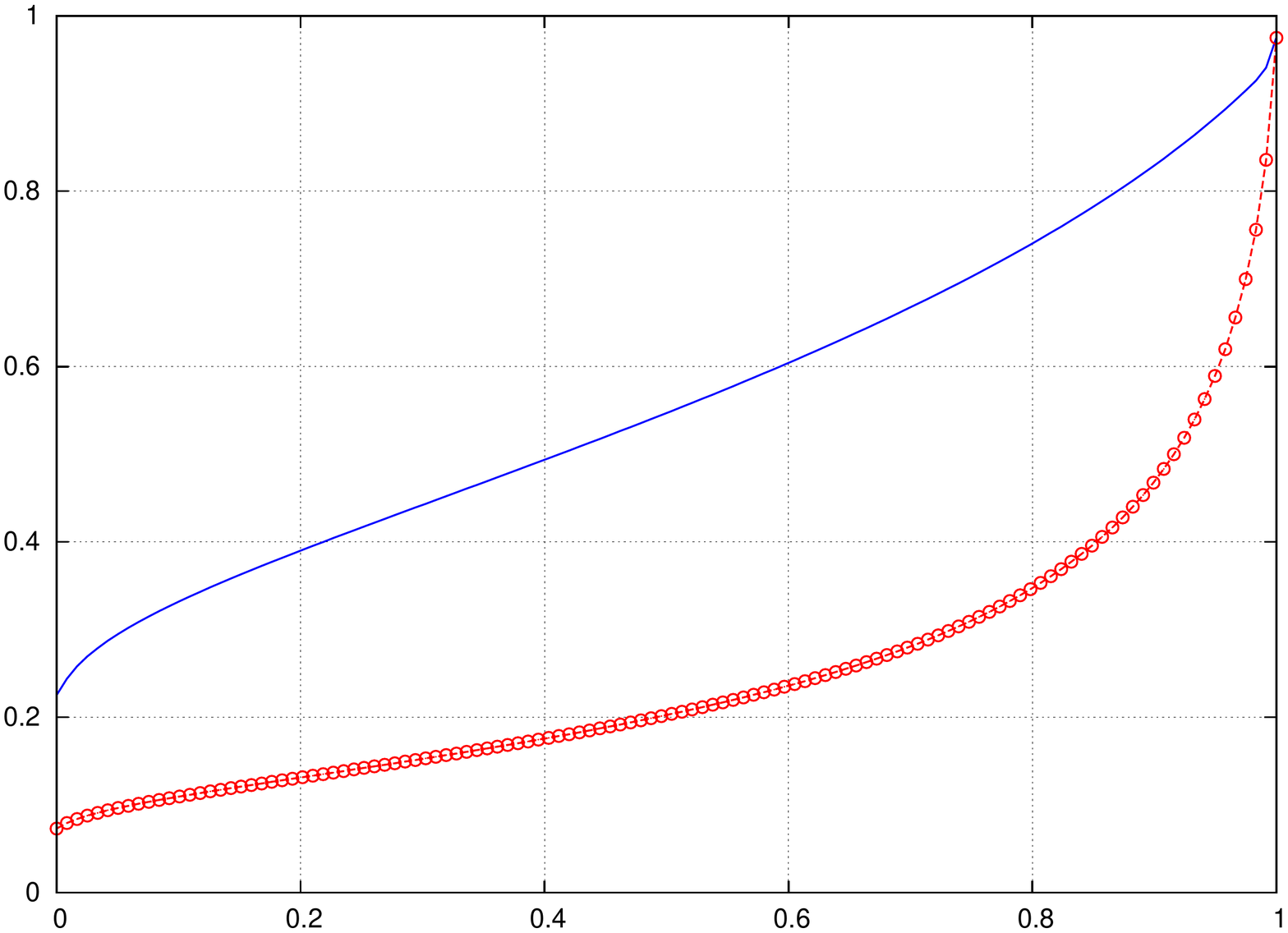}} \\
\subfloat[][Set 4: $f_1(x,y)$]{\label{fig_set4_2d}\includegraphics[height =6cm]{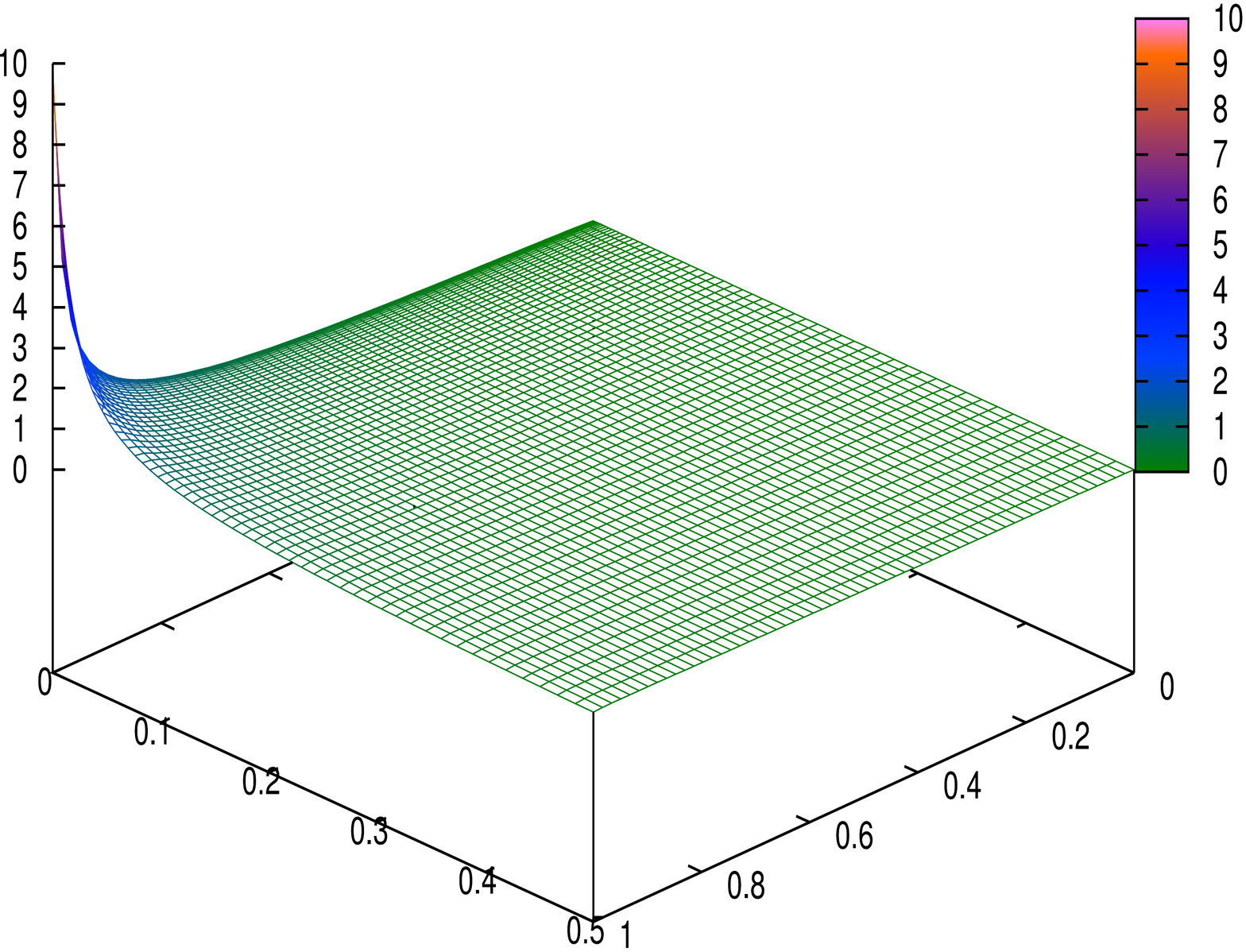}} 
\subfloat[][Set 4:$\; f_2(x)$ (solid) and $f_3(x)$ (circles)]{\label{fig_set4_1d}\includegraphics[height =6cm]{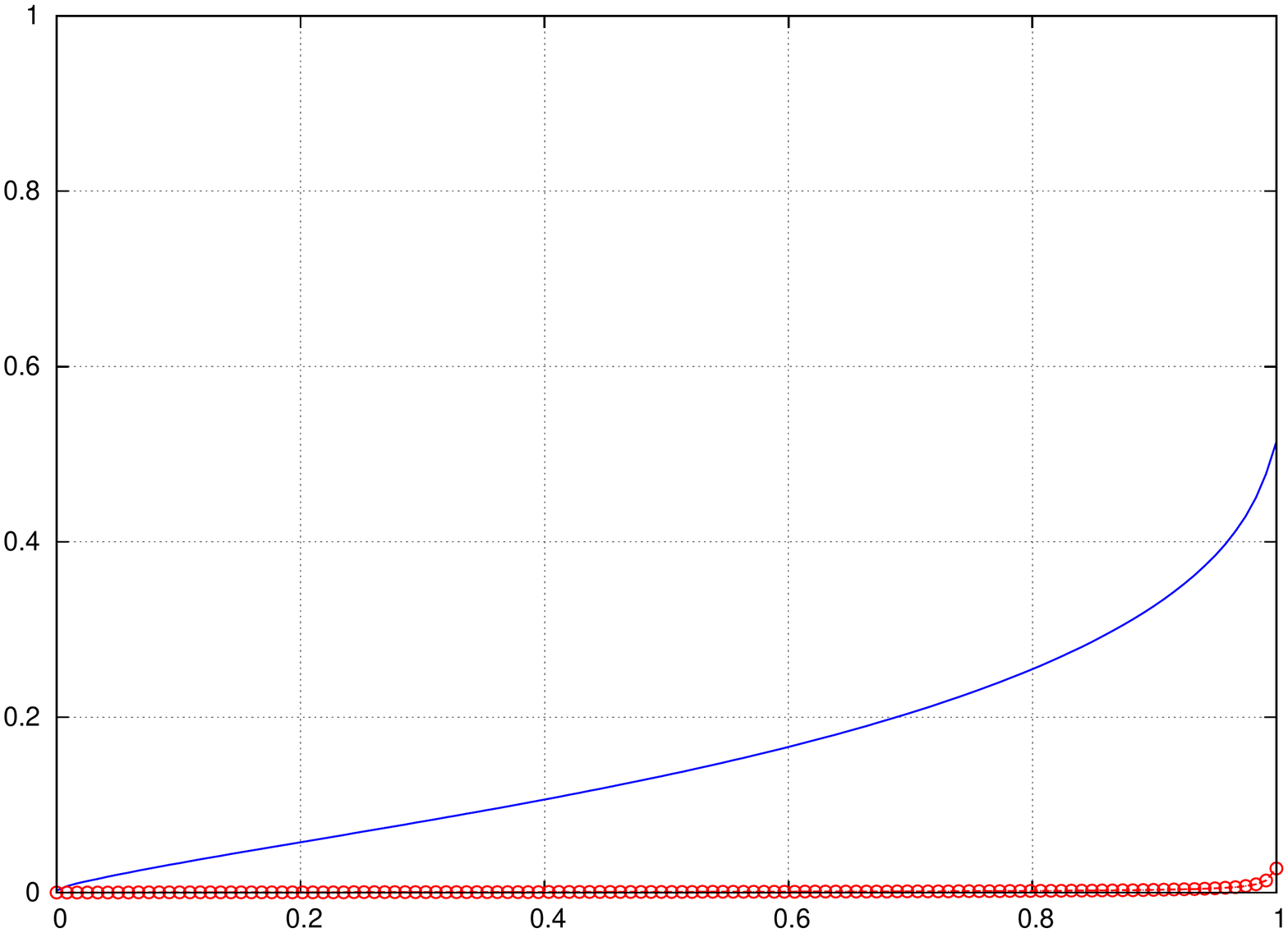}} 
\caption{Bounded variation case ($\sigma=0$): computing the density of the overshoot $f_1(x,y)$ ($x\in (0,1)$, $y\in (0,0.5)$),  probability of first exit $f_2(x)$
and probability of creeping $f_3(x)$ for parameter Set 3 (positive drift $\mu=1$) and Set 4 (negative drift $\mu=-1$).} 
\label{fig_Set34}
\end{figure}

\begin{figure}
\centering
\subfloat[][Set 1]{\label{fig_Vdxds_set1}\includegraphics[height =6cm]{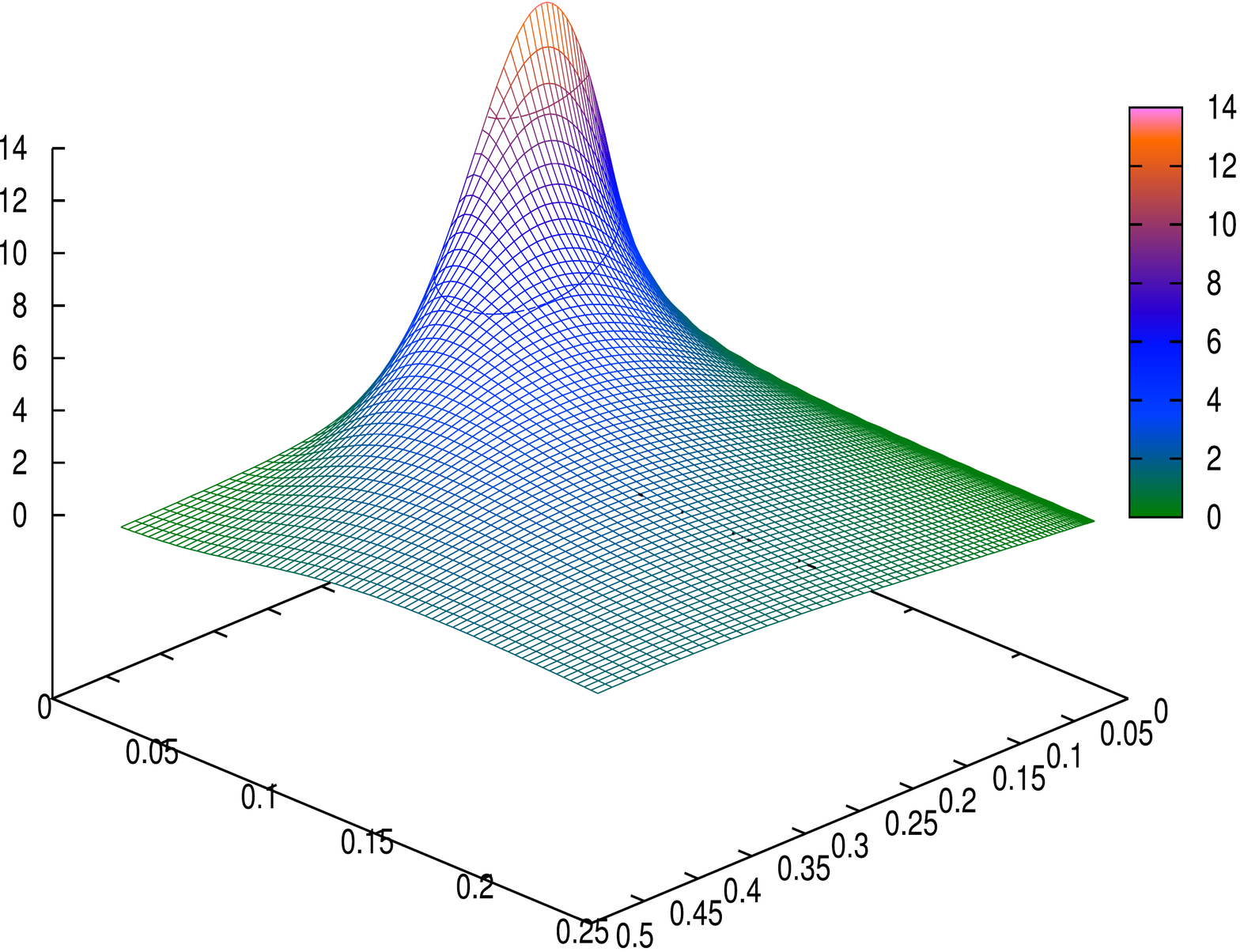}} 
\subfloat[][Set 2]{\label{fig_Vdxds_set2}\includegraphics[height =6cm]{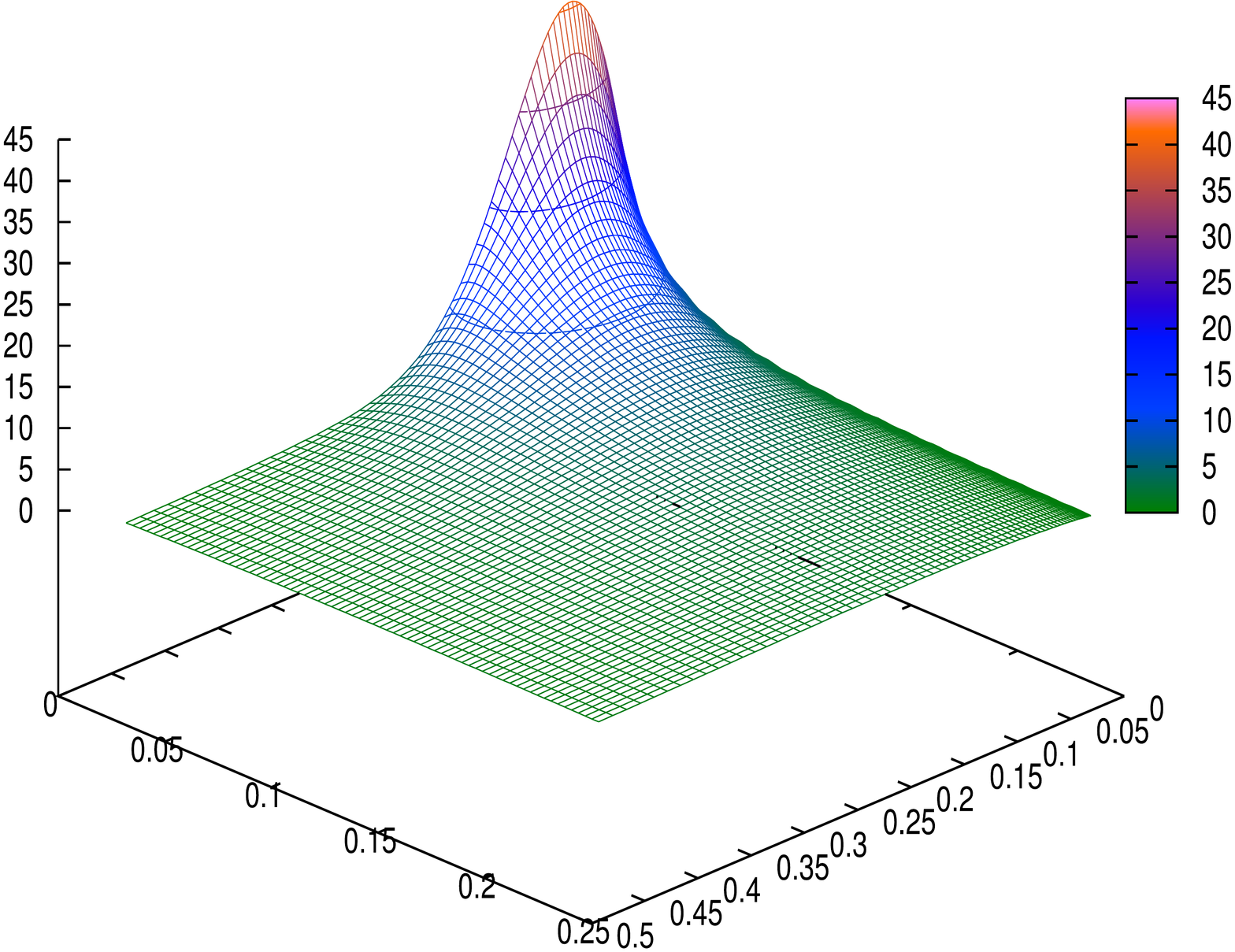}} \\
\subfloat[][Set 3]{\label{fig_Vdxds_set3}\includegraphics[height =6cm]{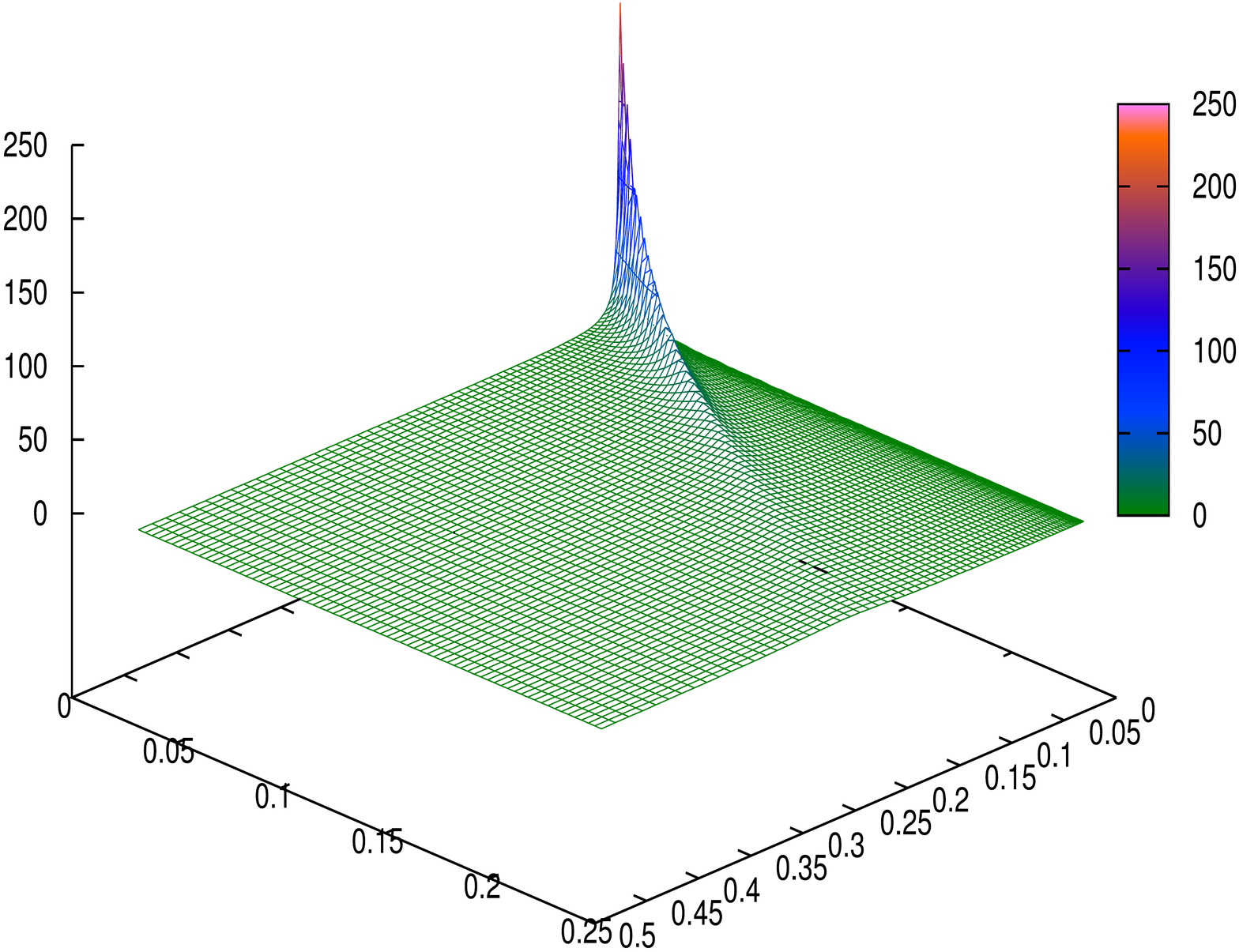}} 
\subfloat[][Set 4]{\label{fig_Vdxds_set4}\includegraphics[height =6cm]{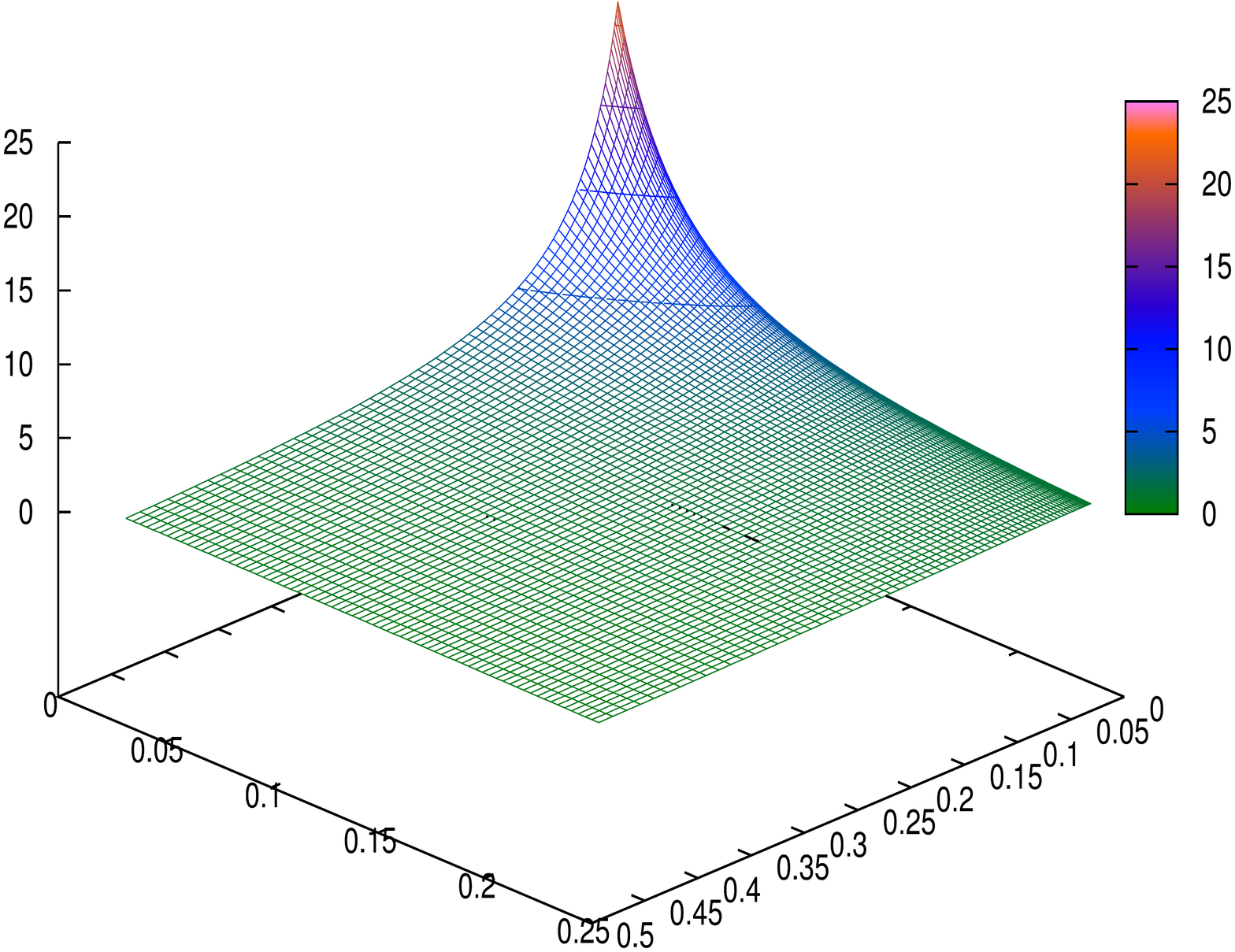}} 
\caption{Computing the density of the renewal measure ${\mathcal U}(\d s, \d x)$,  $s \in (0,0.25)$ and $x\in (0,0.5)$.} 
\label{fig_Vdxds}
\end{figure}

\end{document}